\documentclass[10pt,a4paper,oneside]{amsart}
\usepackage{amsmath, amssymb, amsfonts, amsthm}
\usepackage[a4paper]{geometry}
\numberwithin{equation}{section}

\usepackage[Symbol]{upgreek}
\renewcommand{\epsilon}{\varepsilon}
\renewcommand{\phi}{\varphi}
\renewcommand{\theta}{\vartheta}
\usepackage{stmaryrd}
\usepackage[scr=dutchcal]{mathalfa}
\usepackage{mathtools}

\DeclareMathOperator{\ab}{ab}
\DeclareMathOperator{\Aut}{Aut}

\DeclareMathOperator{\diag}{diag}
\DeclareMathOperator{\End}{End}
\DeclareMathOperator{\Frac}{Frac}
\DeclareMathOperator{\Gal}{Gal}
\DeclareMathOperator{\GL}{GL}

\DeclareMathOperator{\im}{im}
\DeclareMathOperator{\ind}{ind}
\DeclareMathOperator{\Irr}{Irr}
\DeclareMathOperator{\nr}{nr}
\let\orr\relax
\DeclareMathOperator{\orr}{or}

\DeclareMathOperator{\pd}{pd}
\DeclareMathOperator{\rad}{rad}
\DeclareMathOperator{\res}{res}
\DeclareMathOperator{\Quot}{Quot}

\newcommand{\G}{\mathcal G}
\newcommand{\HH}{\mathcal H}
\newcommand{\OO}{\mathcal O}
\newcommand{\Q}{\mathcal Q}

\newcommand{\QQ}{\mathbb Q}

\newcommand{\ZZ}{\mathbb Z}

\newcommand{\al}{\mathrm{c}}
\newcommand{\cent}{\mathfrak z}
\newcommand{\cyc}{\mathrm{cyc}}

\newcommand{\tD}{\widetilde D}
\newcommand{\ff}{\mathbf f}
\newcommand{\ii}{\mathbf i}
\newcommand{\jj}{\mathbf j}
\newcommand{\T}{\mathbf T}
\newcommand{\w}{\mathbf w}
\newcommand{\X}{\mathbf X}
\newcommand{\vv}{\mathbf v}
\newcommand{\kk}{\mathbf k}

\newcommand{\1}{\mathbf 1}
\usepackage{bbm}
\newcommand{\f}{\mathbbm f}
\usepackage{amsbsy}
\newcommand{\ggamma}{\boldsymbol \gamma}
\newcommand{\ddelta}{\boldsymbol \delta}
\newcommand{\ttau}{\boldsymbol \tau}
\renewcommand{\ll}{{\boldsymbol \ell}}
\newcommand{\normal}{\mathrel{\unlhd}}

\usepackage{tikz-cd}
\usepackage{colonequals}

\usepackage[style=alphabetic,backend=biber,url=false,maxbibnames=9,isbn=false]{biblatex}
\usepackage{enumitem}

\usepackage[hidelinks,plainpages=false,pdfpagelabels]{hyperref}
\usepackage[capitalise, noabbrev]{cleveref}

\theoremstyle{plain}
\newtheorem*{theorem*}{Theorem}
\newtheorem{theorem}{Theorem}[section]
\newtheorem{lemma}[theorem]{Lemma}
\newtheorem{proposition}[theorem]{Proposition}
\newtheorem{corollary}[theorem]{Corollary}

\theoremstyle{definition}
\newtheorem{definition}[theorem]{Definition}
\theoremstyle{remark}
\newtheorem{remark}[theorem]{Remark}
\newtheorem{example}[theorem]{Example}

\newlist{propositionlist}{enumerate}{1}
\setlist[propositionlist]{label=(\roman{propositionlisti}), ref=\thetheorem.\roman{propositionlisti},noitemsep, topsep=.2ex}
\Crefname{propositionlisti}{Proposition}{Propositions}

\newlist{lemmalist}{enumerate}{1}
\setlist[lemmalist]{label=(\roman{lemmalisti}), ref=\thetheorem.\roman{lemmalisti},noitemsep, topsep=.2ex}
\Crefname{lemmalisti}{Lemma}{Lemmata}

\newlist{corollarylist}{enumerate}{1}
\setlist[corollarylist]{label=(\roman{corollarylisti}), ref=\thetheorem.\roman{corollarylisti},noitemsep, topsep=.2ex}
\Crefname{corollarylisti}{Corollary}{Corollaries}

\Crefname{lemma}{Lemma}{Lemmata}
\Crefname{claim}{Claim}{Claims}
\Crefname{proposition}{Proposition}{Propositions}
\Crefname{conjecture}{Conjecture}{Conjectures}
\Crefname{example}{Example}{Examples}
\crefname{page}{page}{pages}
\Crefname{condition}{Condition}{Conditions}
\Crefname{question}{Question}{Questions}

\NewDocumentEnvironment{noproof}{m}{
  \par\pushQED{\qed}\UseName{#1}%
}{\popQED\UseName{end#1}}

\title{Total rings of quotients of higher dimensional Iwasawa algebras} \author{} \date{}
\author{Ben Forrás}
\address{University of Ottawa\\
	Department of Mathematics and Statistics \\
	STEM Complex \\
	150 Louis-Pasteur Pvt\\
	Ottawa, ON \\
	Canada K1N 6N5}
\email{bforras@uottawa.ca}
\urladdr{https://bforras.eu}

\date{Version of 2026-09-24}

\begin{document}

\begin{abstract}
    Let $\G=H\rtimes\Gamma$ be a semidirect product of a finite group $H$ with $\Gamma=\ZZ_p^d$, the direct product of finitely many copies of the additive group of the $p$-adic integers. 
    In the case $d=1$, Iwasawa algebras of such profinite groups were studied by Ritter--Weiss, Lau, Nickel, and the author. 
    We generalise these results to all $d\ge1$, describing the Wedderburn decomposition of the total ring of quotients of the Iwasawa algebra of $\G$.
\end{abstract}

\maketitle

\section{Introduction}

\subsection{Overview}
Let $p$ be an odd prime number and $d\ge1$ an integer.
Let $\Gamma=\ZZ_p^d$ be the direct product of $d$ copies of the additive group of the $p$-adic integers, let $H$ be a finite group, and consider a semidirect product $\G=H\rtimes\Gamma$, which is a $d$-dimensional $p$-adic Lie group. Let $\Lambda(\G)\colonequals\ZZ_p\llbracket\G\rrbracket$ be the completed group ring of $\G$ over $\ZZ_p$, and let $\Q(\G)\colonequals\Quot(\Lambda(\G))$ be its total ring of quotients, obtained by inverting all central regular elements of $\Lambda(\G)$. 

If the semidirect product $H\rtimes\Gamma$ is in fact a direct product, then there is an isomorphism \[\Lambda(\G)=\Lambda(\Gamma\times H)\simeq\Lambda(\Gamma)[H]\simeq\ZZ_p[[X_1,\ldots,X_d]][H],\] so all properties of $\Q(\G)$ are determined by those of $\QQ_p[H]$ via base change. Note that the group ring $\QQ_p[H]$ is well-understood due to the work of Hasse \cite{Hasse}. For semidirect products, the problem of describing $\Q(\G)$ is significantly more subtle.

In the case $d=1$, the ring $\Q(\G)$ was studied by Ritter--Weiss \cite{TEIT-II}, who inter alia showed that $\Q(\G)$ is a semisimple algebra. For $d=1$ and $H$ a $p$-group, Lau \cite{Lau} provided an initial description of the skew fields occurring in the Wedderburn decomposition of $\Q(\G)$. For $d=1$ and $H$ an arbitrary finite group, Nickel \cite[\S1]{NickelConductor} characterised the centres of the skew fields in the Wedderburn decomposition, and proved a divisibility relation involving their indices. A complete description of the Wedderburn decomposition of $\Q(\G)$ in terms of $\QQ_p[H]$ was achieved by the author \cite{W,W2}. 

Namely, in the case $d=1$, the skew fields occurring in the Wedderburn decomposition of $\Q(\G)$ contain a unique maximal order of the form $\OO_D[[X;\tau,\tau-1]]$. Here $\OO_D$ is the unique maximal order in a skew field $D$ that is a skew field occurring in the Wedderburn decomposition of $\QQ_p[H]$, $\tau\in\Aut(D)$ is a certain automorphism, and $\OO_D[[X;\tau,\tau-1]]$ is the ring of formal skew power series. The abelian group underlying $\OO_D[[X;\tau,\tau-1]]$ agrees with that of the formal power series ring $\OO_D[[X]]$, and multiplication is defined by $Xt=\tau(t)X+\tau(t)-t$ for all $t\in\OO_D$.

This result serves as the starting point for investigating $\Q(\G)$ for $d\ge1$. As $\Gamma=\ZZ_p^d$ is a direct product, the individual copies of $\ZZ_p$ have commuting actions on $H$. Intuitively, since one copy of $\ZZ_p$ leads to a one-variable skew power series ring, a direct product of $d$ copies of $\ZZ_p$ should lead to a skew power series ring in $d$ variables, where the variables commute with each other. Making this idea precise leads to some technical difficulties: the aim of this article is to tackle these. The main result is the following; see \cref{Wedderburn-general} for a more explicit statement, and \cref{sec:skew-power-series-def} for the definition of multivariate skew power series rings.

\begin{theorem*}
    The skew fields occurring in the Wedderburn decomposition of $\Q(\G)$ are of the form
    \[\Quot\left(\OO_D[[X_1,\ldots,X_d;\, \tau_1,\ldots,\tau_d,\, \tau_1-1,\ldots,\tau_d-1]]\right),\]
    where $D$ is a finite dimensional skew field over a local field, $\OO_D$ is its maximal order, $\tau_i\in\Aut(D)$ are automorphisms, $X_i t=\tau_i(t)X_i+\tau_i(t)-t$ for $t\in \OO_D$, and $X_iX_j=X_jX_i$ for $1\le i,j\le d$.
\end{theorem*}
Moreover, the indices and sizes of the matrix rings occurring in the Wedderburn decomposition are explicitly related to those in the decomposition of $\QQ_p[H]$, see \cref{s-n-general}.

Most results of this work are generalisations of analogous results in \cite{TEIT-II}, \cite{NickelConductor}, and especially \cite{W} and \cite{W2}, which deal with the $1$-dimensional case. It is necessary to retrace the steps taken in these works: the $d$-dimensional statements do not follow directly from the $1$-dimensional ones. We shall not repeat proofs that carry over to higher dimensions without modification, instead focusing on the parts that require nontrivial adjustments. There are many references to corresponding statements in the works cited above, and the reader is invited to consult these for further explanations.

\subsection{Outline}
In \cref{sec:basics}, we investigate basic properties of $\Q(\G)$, relating Wedderburn components to irreducible characters of $\G$ and $H$.
\Cref{sec:skew-power-series} is devoted to discussing multivariate skew power series rings.
In \cref{sec:Wedderburn}, we derive a description of the Wedderburn decomposition of $\Q(\G)$, building upon the results of the previous two sections.
\Cref{sec:Iwasawa} presents a brief outlook towards applications in Iwasawa theory.

\subsection{Notation and conventions}
We use lowercase bold-face letters to denote $d$-tuples.
The word `ring' means not necessarily commutative ring with $1$.
A domain is a ring with no zero divisors, and an integral domain is a commutative domain.
We abuse notation by writing $\oplus$ for a direct product of rings, even though this is not a coproduct in the category of rings.
The centre of a group or a ring is denoted by $\cent(-)$.

\subsection{Acknowledgements}
The author thanks Andreas Nickel for helpful conversations as well as for his comments on an earlier version of the manuscript.
This work was partially funded by the Deutsche Forschungsgemeinschaft (DFG, German
Research Foundation) under project number 559516518.

\section{Representation theory of \texorpdfstring{$\Q(\G)$}{Q(G)}} \label{sec:basics}
Let $\G\simeq H\rtimes \Gamma$, where $H$ is a finite group, and $\Gamma\simeq\ZZ_p^d$; then $\G$ is a $p$-adic Lie group of dimension $d\ge1$. Let $F/\QQ_p$ be a finite extension with ring of integers $\OO_F$, and let $\Lambda^{\OO_F}(\G)\colonequals \OO_F\llbracket\G\rrbracket$ be the associated Iwasawa algebra. Let $\Q^F(\G)\colonequals \Quot(\Lambda^{\OO_F}(\G))$ be its total ring of quotients.

\subsection{Semisimplicity} \label{sec:semisimplicity}
Suppose that topological generators $\gamma_1,\ldots,\gamma_d$ are given in each factor $\ZZ_p$ of $\Gamma$.
We fix some notation for working with tuples in $\Gamma$.
Write $\ggamma\colonequals(\gamma_1,\ldots,\gamma_d)\in\Gamma$.
For $\ii=(i_1,\ldots,i_d)\in\ZZ^d$, we let $\ggamma^\ii\colonequals(\gamma_1^{i_1},\ldots,\gamma_d^{i_d})\in\Gamma$. We will also write $p^\ii\colonequals(p^{i_1},\ldots,p^{i_d})\in\ZZ^d$. If $\ii,\ii'\in \ZZ^d$, then $\ii\le\ii'$ is understood entrywise, that is, $i_j\le i_j'$ for all $1\le j\le d$. The notation $\mathbf 0\le \ii$ shall mean that all entries of $\ii$ are non-negative. Similarly, $\ii+\ii'$ is the entrywise sum, that is, its $j$th entry is $i_j+i'_j$. Finally, we define $\Pi(\ii)\colonequals\prod_{j=1}^d i_j\in \ZZ$.

\begin{lemma} \label{Gamma0}
    There exists a tuple $\mathbf n_0\in\ZZ_{\ge0}^d$ such that the subgroup 
    \[\Gamma_0 = \prod_{j=1}^d \overline{\left\langle\gamma_j^{p^{n_{0,j}}}\right\rangle}=\prod_{j=1}^d \overline{\left\langle\gamma_{0,j}\right\rangle}\]
    is central in $\G$, where for $1\le j\le d$, $\gamma_{0,j}\colonequals\gamma_j^{p^{n_{0,j}}}$, and $\overline{\left\langle\gamma_{0,j}\right\rangle}$ is the procyclic group topologically generated by $\gamma_{0,j}$.
\end{lemma}
\begin{proof}
    Let $\phi:\Gamma\to\Aut(H)$ denote the automorphism describing the action of $\Gamma$ on $H$ in the semidirect product $\G=H\rtimes_\phi\Gamma$: that is, $g\cdot h=\phi(g)(h)\cdot g$ for all $g\in \Gamma$ and $h\in H$. Since $H$ is a finite group, so is $\Aut(H)$, hence $\ker\phi\subseteq \Gamma$ is open. In the product topology on $\Gamma$, the basic open subgroups of $\Gamma$ are those of the form $\prod_{j=1}^d\overline{\left\langle\gamma_j^{i_j}\right\rangle}$ with $\ii\in\ZZ_{\ge0}^d$. Since $\ker\phi$ is open, it contains such a basic open subgroup, which is then central in $\G$, as desired.
\end{proof}
From now on, we fix such a tuple $\mathbf n_0=(n_{0,1},\ldots,n_{0,d})$. This yields the following decomposition of the Iwasawa algebra:
\begin{equation} \label{eq:Lambda-G-decomp}
    \Lambda^{\OO_F}(\G)=\bigoplus_{g\in\G/\Gamma_0}\Lambda^{\OO_F}(\Gamma_0)g=\bigoplus_{\mathbf 0\le\ii< p^{\mathbf n_0}} \Lambda^{\OO_F}(\Gamma_0)[H]\ggamma^\ii
\end{equation}
where $p^{\mathbf n_0}=(p^{n_{0,1}},\ldots,p^{n_{0,d}})$.

\begin{proposition} \label{semisimplicity}
    The ring $\Q^F(\G)$ is semisimple artinian.
\end{proposition}
\begin{proof}
    In the case $d=1$, this is due to Ritter--Weiss \cite[Proposition~5(1)]{TEIT-II}. Their proof readily generalises to higher dimensions. Indeed, \eqref{eq:Lambda-G-decomp} yields
    \begin{equation} \label{eq:QG-pn0}
        \Q^F(\G)=\bigoplus_{\mathbf 0\le\ii< p^{\mathbf n_0}} \Q^F(\Gamma_0)[H]\ggamma^\ii,
    \end{equation}
    thus showing that $\Q^F(\G)$ is a finite dimensional $\Q^F(\Gamma_0)$-algebra.

    Let $\mathfrak l\subseteq \Q^F(\G)$ be a nilpotent left ideal. We show that $\mathfrak l=0$, which implies semisimplicity of $\Q^F(\G)$ by \cite[\nopp $(5.15)$ and $(5.18)$]{CR}. 
    Since $\Q^F(\G)$ is noetherian, $\mathfrak l$ has a finite generating set, say $y_1,\ldots,y_n$.
    Clearing denominators, we find a nonzero element $x\in\Lambda^{\OO_F}(\Gamma_0)$ such that $xy_1,\ldots,xy_n$ are all contained in $\Lambda^{\OO_F}(\G)$.
    
    For $\ii\ge0$, let 
    \[\Gamma_0^\ii\colonequals \prod_{j=1}^d\overline{\left\langle\gamma_{0,j}^{i_j}\right\rangle};\]
    note that $\Gamma_0^\ii$ is a subgroup of $\G$ of finite index.
    Let $\Delta\Gamma_0^\ii$ be the corresponding augmentation ideal defined by the short exact sequence $\Delta\Gamma_0^\ii\hookrightarrow \Lambda^{\OO_F}(\Gamma_0^\ii)\twoheadrightarrow\OO_F$.
    For each $k=1,\ldots,n$, we have the following:
    \[xy_k \pmod{\Delta\Gamma_0^\ii} \in \Lambda^{\OO_F}(\G/\Gamma_0^\ii)=\OO_F[\G/\Gamma_0^\ii]\subseteq F[\G/\Gamma_0^\ii].\]
    Let $\mathfrak l'$ be the ideal of $F[\G/\Gamma_0^\ii]$ generated by $xy_1,\ldots,xy_n$.
    Since $x$ is central and $\mathfrak l$ is nilpotent, $\mathfrak l'$ is also nilpotent. 
    But the group ring $F[\G/\Gamma_0^\ii]$ is semisimple by Maschke's theorem, and therefore has no nontrivial nilpotent ideals. Therefore $\mathfrak l'=0$, so $xy_1=\ldots=xy_n=0$, which proves $\mathfrak l=0$.
\end{proof}

\begin{remark}
    Note that Ritter--Weiss's proof \cite[Proposition~5(1)]{TEIT-II} of semisimplicity in the one-dimensional case contains an inaccuracy. Namely, they claim the existence of a nonzero $x\in\Lambda^{\OO_F}(\Gamma_0)$ such that $x\mathfrak l \subset \Lambda^{\OO_F}(\G)$. But since $\mathfrak l$ is an ideal of $\Q^F(\G)$, the ideal $x\mathfrak l$ will still contain denominators for any such $x$. We thank Andreas Nickel for pointing this out to us.
\end{remark}

\subsection{Invariants attached to characters} \label{sec:invariants}
For a profinite group $G$, let $\Irr(G)$ denote the set of irreducible characters of $G$ with open kernel.
There is a left group action of $\G$ on $\Irr(H)$ given by ${}^g\eta(h)\colonequals\eta(g h g^{-1})$ for $\eta\in\Irr(H)$, $g\in \G$, and $h\in H$.

Let $\eta\in\Irr(H)$, and let $\G_\eta\le\G$ denote the stabiliser of $\eta$ with respect to the action just defined: this is a non-empty open subgroup of $\G$. Since $\eta$ is a class function on $H$, it is clear that $H\le\G_\eta$. Therefore the quotient $\G/\G_\eta$ is isomorphic to a quotient of $\Gamma$ by an open subgroup. Since $\Gamma$ is an abelian group, there is a left group action of $\Gamma$ on $\Irr(H)$ given by ${}^g\eta(h)\colonequals\eta(g^{-1} h g)$ for $\eta\in\Irr(H)$, $g\in \Gamma$, and $h\in H$. Notice the difference between this and the action of $\G$ defined above; since $\G$ is in general non-abelian, the same formula would not define a left action of $\G$ on $\Irr(H)$. From now on, an action of $\Gamma$ or $\G/\G_\eta$ shall always mean the latter action: this fits better with the convention of writing coefficients on the left and group elements on the right in a group ring.

For the rest of \cref{sec:basics}, fix a character $\chi\in\Irr(\G)$. As we shall see below, each such $\chi$ corresponds to a simple component of $\Q^F(\G)$.
Let $\eta\mid\res^\G_H\chi$ be an irreducible constituent of its restriction to $H$.
In the $\chi$-component of $\Q^F(\G)$, the action of $\Gamma$ on $H$ is governed by the action of $\G$ on $\eta$. Of special importance will be the elements of $\Gamma$ that act by a Galois automorphism on $\eta$.

A key difference between the one-dimensional and multidimensional cases is that in the latter, there can be nontrivial relations between the actions of distinct basis elements $\gamma_i\ne \gamma_j$: in fact, two such elements may even have the same action on $H$ and hence on $\eta$.
This behaviour is undesirable when one studies Galois actions on $\eta$: to remedy this, we will need to choose a basis with appropriate properties.

\subsubsection{Open subgroups of $\Gamma$} \label{sec:open-subgroups}
We now describe such open subgroups. For $d=1$, such a description is easy and well-known: all open subgroups of $\ZZ_p$ are of the form $p^r\ZZ_p$ for some $r\ge0$. A possible generalisation of this fact to higher dimension is as follows.

\begin{lemma}  \label{ZPID}
    If $U$ is a non-empty open subgroup of $\Gamma$, then
    \[\Gamma/U\simeq \prod_{i=1}^d \ZZ_p/p^{b_i}\ZZ_p\]
    for some non-negative integers $b_1,\ldots,b_d$.
\end{lemma}
\begin{proof}
    Let $U\subseteq \Gamma$ be a non-empty open subgroup. View $\Gamma$ as a free $\ZZ_p$-module of rank~$d$. 
    Since $U$ is open, the quotient $\Gamma/U$ is finite, therefore $U$ is also a $\ZZ_p$-module of rank~$d$. 
    As $\ZZ_p$ is a principal ideal domain, the invariant factor theorem \cite[\nopp 4.14]{CR} allows us to choose a $\ZZ_p$-basis $\gamma_1,\ldots,\gamma_d$ of $\Gamma$ such that $\gamma_1^{a_1},\ldots,\gamma_d^{a_d}$ is a basis of $U$ for some $a_1,\ldots,a_d\in\ZZ_p$. 
    Possibly multiplying by some $p$-adic units, we may assume that for $1\le i\le d$, we have $a_i=p^{b_i}$ for some $b_i\in\ZZ_{\ge0}$. The claim follows.
\end{proof}

\begin{remark}
    The proof of \cref{ZPID} above involves choosing a basis suitable for describing a given open subgroup up to isomorphism; we refer to \cref{ex:two-same-action} to illustrate why working with an arbitrary basis is undesirable. 
    Note that the basis depends on the character $\chi$, which will remain fixed throughout this section.
    This will be sufficient for our applications, but we remark that this approach would not quite suffice if a basis were already given, or if we would need to handle multiple open subgroups simultaneously, or if we would need to determine the actual subgroup not only up to isomorphism.
    
    In these cases, instead of using the $\ZZ_p$-module structure, one may work with the group structure.
    In group theory, Goursat's lemma describes the subgroups of two groups. Bauer--Sen--Zvengrowski \cite{Goursat-profinite} proved a generalised version for a product of $d$ groups (Theorem~3.2 of op.cit.), as well as a version for a product of two profinite groups (Proposition~4.2 of op.cit.). Combining their results provides an explicit albeit somewhat cumbersome way to describe subgroups of a product of finitely many profinite groups. This description is in terms of certain subgroups and homomorphisms between them: in the module-theoretic approach, these auxiliary data are subsumed in the choice of a suitable basis.
\end{remark}

Since $H\le\G_\eta$, it follows from \cref{ZPID} that there are $p$-power integers $w_{\chi,i}$ such that
\[\G/\G_\eta\simeq \prod_{i=1}^d \ZZ_p/w_{\chi,i}\ZZ_p\] 
Let us write $\w_\chi\colonequals(w_{\chi,i})$, so that $w_\chi\colonequals\Pi(\w_\chi)=[\G:\G_\eta]$. Note that $w_{\chi,i}=[\overline{\langle\gamma_i\rangle}: \overline{\langle\gamma_i\rangle}\cap \G_\eta]$.
As in \cite[553]{TEIT-II}, we note that $\G_\eta$ depends only on $\chi$, and therefore so does $\w_\chi$, so this notation is justified.

\subsubsection{Triviality of Clifford multipliers}
Recall that we fixed a character $\chi\in\Irr(\G)$. Let $\eta\mid\res^\G_H\chi$ be an irreducible constituent thereof.
Clifford theory, in particular \cite[Proposition~11.4]{CR}, asserts that there is a $z_\chi\in\ZZ_{\ge0}$ such that
\begin{equation}  \label{eq:res-z-eta}
    \res^\G_H\chi = z_\chi \sum_{g\in \G/\G_\eta} {}^g\eta = z_\chi \sum_{\mathbf 0\le\ii< \w_\chi} {}^{\ggamma^\ii} \eta.
\end{equation}
In particular, we have $\chi(1)=z_\chi w_\chi \eta(1)$.
We will refer to $z_\chi$ as the Clifford multiplier of $\chi$ (some authors use the term `ramification index').

\begin{proposition} \label{zchi=1}
    For every $\chi\in\Irr(\G)$, the Clifford multiplier is $z_\chi=1$.
\end{proposition}
\begin{proof}
    All objects appearing in this proof depend on $\chi$, but this will be suppressed from the notation.     
    Let $N_0\colonequals\G/\ker\chi$ be the maximal finite quotient of $\G$ over which $\chi$ factors, and let $\overline H$ be the image of $H$ in $N_0$ under the natural projection map. 
    Since $\G=\Gamma\ltimes H$, we have $N_0=(C_1\times\ldots\times C_d)\ltimes \overline H$ for some cyclic $p$-groups $C_1,\ldots,C_d$. 

    Let $V_0$ be an irreducible $\QQ_p^\al$-representation of $N_0$ affording $\chi$. For $i=1,\ldots,d$, let $N_i\colonequals(C_1\times\ldots\times C_{d-i})\ltimes \overline H$. Note that 
    \[\overline H=N_d\normal N_{d-1}\normal \ldots\normal N_1\normal N_0.\]
    Let $V_i$ be a homogeneous component of $V_{i-1}$ with respect to $N_i$. By induction, we will show that for each $i=1,\ldots,d$, $V_i$ is an irreducible $N_i$-representation; the claim clearly holds for $i=0$.

    Let $G_i$ be the stabiliser of $V_i$ in $N_{i-1}$. Since $V_{i-1}$ is an irreducible representation of $N_{i-1}$ by the inductive hypothesis, $V_i$ is an irreducible representation of $G_i$ by Clifford theory \cite[Satz~V.17.3(e)]{Huppert}.
    Since $N_{i-1}/N_i\simeq C_i$ is cyclic, and since $N_i\normal G_i\normal N_{i-1}$, the quotient $N_{i-1}/G_i$ is also cyclic. Hence the Schur multiplier $H^2(N_{i-1}/G_i,\QQ_p^{\al,\times})$ is trivial, and therefore $V_i$ is also irreducible as an $N_i$-representation by \cite[Satz~V.17.12]{Huppert} or \cite[Corollary~11.47]{CR}.

    Let $\eta\mid\res^\G_H \chi$ be an irreducible constituent, and let $e(\eta)\in F(\eta)[H]$ denote the central primitive idempotent corresponding to $\eta$.
    Then the argument above shows in particular that $e(\eta) V_0$ is an irreducible representation of $H$. On the other hand, \eqref{eq:res-z-eta} shows that this is also the $z_\chi$-fold direct sum of the irreducible representation of $H$ affording $\eta$. It follows that $z_\chi=1$.
\end{proof}

The following is an immediate consequence of \cref{zchi=1}.
\begin{corollary} \label{cor:zchi=1}
    If $\chi\in\Irr(\G)$, and $\eta$ is any irreducible constituent of $\res^\G_H\chi$, then $\langle \eta, \res^\G_H\chi\rangle_H=1$. \qed
\end{corollary}

\subsubsection{Galois orbits of characters}
For $\chi\in\Irr(\G)$ and $\eta\mid\res^\G_H\chi$ an irreducible constituent, and for $F/\QQ_p$ a finite extension, define the fields $F(\eta)\colonequals F(\eta(h):h\in H)$ and $F_\chi\colonequals F(\chi(h):h\in H)$. Then $F(\eta)/F_\chi$ is a finite Galois extension. Note that $F(\eta)$ depends only on $\chi$ and not on the choice of $\eta$, and therefore there is a group action of $\Gal(F(\eta)/F_\chi)$ on the irreducible characters of $H$ dividing $\res^\G_H\chi$: for $\sigma\in\Gal(F(\eta)/F_\chi)$ and $\psi\mid\res^\G_H\chi$ irreducible, define ${}^\sigma\psi\in\Irr(H)$ by ${}^\sigma\psi(h)\colonequals \sigma(\psi(h))$ for $h\in H$.
\begin{lemma} \label{A1.1}
    The following hold:
    \begin{enumerate}[label=(\roman*), ref=(\roman*)]
        \item \label{item:A1.1-etas} There is a positive integer $t_\chi\ge1$ and there are irreducible characters $\eta=\eta_1,\eta_2,\ldots,\eta_{t_\chi}\in \Irr(H)$ such that
        \[\res^\G_H\chi = \sum_{i=1}^{t_\chi} \sum_{\sigma\in\Gal(F(\eta)/F_{\chi})} {}^\sigma\eta_i.\]
        \item \label{item:A1.1-w} $w_\chi = t_\chi [F(\eta):F_{\chi}]$.
    \end{enumerate}
\end{lemma}
\begin{proof}
    Assertion~\ref{item:A1.1-etas} is a direct generalisation of \cite[Lemma~1.1]{NickelConductor}, and the proof is identical. 
    Assertion~\ref{item:A1.1-w} follows by comparing \ref{item:A1.1-etas} with \eqref{eq:res-z-eta}.
\end{proof}

We generalise the results of \cite[\S4.2]{W} concerning the interplay between the actions of $\Gamma$ and $\Gal(F(\eta)/F_\chi)$ on $\Irr(H)$ to higher dimension.
\begin{definition} \label{def:vv}
    For each $g\in \Gamma$, let $w_{\chi,g}\colonequals [\overline{\langle g\rangle}: \overline{\langle g\rangle} \cap \G_\eta]$, and define
    \[v_{\chi,g}\colonequals \min\left\{0<\ell\le w_{\chi,g} : \exists \tau_g\in\Gal(F(\eta)/F), {}^{g^{\ell}} \eta= {}^{\tau_g}\eta\right\}.\]
\end{definition}
Since $g^{w_{\chi,g}}\in\G_\eta$, it acts trivially on $\eta$, so the set in \cref{def:vv} is non-empty. It also follows that there is a divisibility
\begin{equation} \label{eq:v-divisibilities}
    v_{\chi,g}\mid w_{\chi,g}.
\end{equation}
\begin{lemma} \label{taui-uniqueness-and-order}
    Let $g\in\Gamma$. For $\ell=v_{\chi,g}$ in \cref{def:vv}, the automorphism $\tau_g$ is unique, has order $w_{\chi,g}/v_{\chi,g}$, and $\tau_g\in\Gal(F(\eta)/F_\chi)$.
\end{lemma}
\begin{proof}
    Let $\chi_g\in\Irr(\overline{\langle g\rangle}\ltimes H)$ such that $\chi_g\mid \res^\G_{\overline{\langle g\rangle}\ltimes H} \chi$ and $\eta\mid\res^{\overline{\langle g\rangle}\ltimes H}_H\chi_g$. Observing \cref{def:vv}, $\chi$ and $\chi_g$ define the same $v_{\chi,g}$.
    Applying the one-dimensional result \cite[Proposition~4.5]{W} to $\chi_g$, we obtain that $\langle\tau_g\rangle=\Gal(F(\eta)/F_{\chi_g})$.
    The cited Proposition holds for procyclic groups (i.e. in the case $d=1$), and therefore it is applicable to $\chi_g$.
\end{proof}
From now on, for any $g\in\Gamma$, let $\tau_g$ denote this unique automorphism. 
We now proceed to show that the basis of $\Gamma$ chosen above behaves well with respect to these Galois automorphisms.
In \cref{sec:examples} we will discuss examples showcasing that an arbitrary basis of $\Gamma$ would fail to have this property.

\begin{proposition} \label{cyclicity}
    The $\ZZ_p$-basis $\gamma_1,\ldots,\gamma_d\in\Gamma$ satisfies
    \[\Gal(F(\eta)/F_{\chi}) = \prod_{i=1}^d \langle \tau_{\gamma_i}\rangle.\]
    Moreover, $[F(\eta):F_{\chi}]=w_\chi/v_\chi$ and $t_\chi=v_\chi$.
\end{proposition}
For brevity, we shall write $w_{\chi,i}=w_{\chi,\gamma_i}$, $v_{\chi,i}=v_{\chi,\gamma_i}$, and $\tau_i=\tau_{\gamma_i}$. For a $d$-tuple $\ll=(\ell_1,\ldots,\ell_d)\in\ZZ^d$, let us write $\ttau^{\ll}\colonequals\prod_{i=1}^d \tau_i^{\ell_i}$. Note that the order of multiplication does not matter because $\Gal(F(\eta)/F_{\chi})$ is abelian, as it is a subquotient of the Galois group of a cyclotomic extension of $F$.
\begin{proof}
    We show that 
    \begin{equation} \label{eq:tau-separation}
        \langle\tau_i\rangle\cap\langle\tau_1,\ldots,\tau_{i-1}\rangle=1
    \end{equation}
    as subgroups of $\Gal(F(\eta)/F_\chi)$. 
    Considering the decomposition $\res^\G_H\chi$ coming from \eqref{eq:res-z-eta} and using \cref{zchi=1}, we see that the following is a constituent of it:
    \begin{equation*}
        \sum_{\substack{\jj\in\ZZ^i \\ \mathbf 0\le\jj< \w_\chi}} {}^{\ggamma^\jj} \eta = \sum_{\substack{\kk\in\ZZ^i\\ \mathbf 0\le\kk< \vv_\chi}} \sum_{\substack{\ll\in\ZZ^i\\ \mathbf 0\le \ll< \frac{\w_{\chi}}{\vv_{\chi}}}} {}^{\ttau^{\ll}}\left({}^{\ggamma^{\kk}}\eta\right) \,\Big|\, \res^\G_H\chi.
    \end{equation*}
    Here the inequality $\jj<\w_\chi$ means $j_1<w_{\chi,1}$, $\ldots$, $j_i<w_{\chi,i}$, and similarly the other inequalities are also understood in the first $i$ entries.
    Note that the entry-wise quotient $\w_\chi/\vv_\chi$ is in $\ZZ^d$ by \eqref{eq:v-divisibilities}.
    Now if \eqref{eq:tau-separation} were to fail, then there would be $\mathbf 0\le \ll,\ll'< \w_{\chi}/\vv_{\chi}$ such that $\ttau^\ll=\ttau^{\ll'}$. But then $\langle {}^{\ttau^\ll}\eta,\res^\G_H\chi\rangle_H\ge2$, which would contradict \cref{cor:zchi=1}.
        
    Since \eqref{eq:tau-separation} holds for all $i=1,\ldots,d$, we find that the (outer) direct product in the statement is indeed a subgroup of $\Gal(F(\eta)/F_\chi)$.
    Then we have the following decomposition coming from \eqref{eq:res-z-eta}, using \cref{zchi=1}:
    \begin{equation} \label{eq:res-two-decompositions}
        \res^\G_H\chi = \sum_{\mathbf 0\le\ii< \w_\chi} {}^{\ggamma^\ii} \eta = \sum_{\mathbf 0\le\kk< \vv_\chi} \sum_{\mathbf 0\le \ll< \frac{\w_{\chi}}{\vv_{\chi}}} {}^{\ttau^{\ll}}\left({}^{\ggamma^{\kk}}\eta\right).
    \end{equation}

    We show that the automorphisms $\tau_i$ generate the full Galois group. Suppose that $\tau\in \Gal(F(\eta)/F_\chi)$ but $\tau\notin\prod_{i=1}^d \langle\tau_i\rangle$. 
    Recall that $F(\eta)/F_\chi$ is a $p$-extension: consequently, $\tau$ has $p$-power order, say $p^k$.
    By construction, we have
    \begin{equation} \label{eq:pk-div}
        p^k \,\Big|\, \frac{w_{\chi,i}}{v_{\chi,i}}
    \end{equation}
    for all $i=1,\ldots,d$, where the right hand side is the order of $\tau_i$ by \cref{taui-uniqueness-and-order}.
    Recall from \cref{A1.1}.\ref{item:A1.1-etas} that there is a decomposition
    \begin{equation} \label{eq:res-second-decomposition}
        \res^\G_H\chi = \sum_{i=1}^{t_\chi} \sum_{\sigma\in\Gal(F(\eta)/F_{\chi})} {}^\sigma\eta_i,
    \end{equation}
    where each $\eta_i$ is some $\G$-conjugate of $\eta$. It follows that there is an element $g\in\G$ such that ${}^\tau\eta={}^g\eta$. We can write 
    \begin{equation} \label{eq:g-ai}
        g=\prod_{i=1}^d \gamma_i^{a_i}
    \end{equation}
    for some integers $a_i\ge0$.
    Therefore $\tau^{p^k}=1\in \prod_{i=1}^d\langle\tau_i\rangle$, so 
    \begin{equation} \label{eq:tau-pk}
        \tau^{p^k}=\prod_{i=1}^d \tau_i^{p^k a_i / v_{\chi,i}}=1.
    \end{equation}
    We have the following divisibilities for all $i=1,\ldots,d$:
    \begin{equation} \label{eq:div-div}
        \frac{w_{\chi,i}}{v_{\chi,i}} \,\Big|\, \frac{p^k a_i}{v_{\chi,i}} \,\Big|\, \frac{w_{\chi,i} a_i}{v_{\chi,i}^2},
    \end{equation}
    where the first divisibility comes from \cref{taui-uniqueness-and-order} and \eqref{eq:tau-pk}, and the second one from \eqref{eq:pk-div}. It follows that $v_{\chi,i}\mid a_i$ for all $i=1,\ldots,d$. Then from \eqref{eq:g-ai} we obtain $\tau=\prod_{i=1}^d \tau_i^{a_i/v_{\chi,i}}$, showing the claim.

    The last two statements follow by comparing the decompositions \eqref{eq:res-two-decompositions} and \eqref{eq:res-second-decomposition}.
\end{proof}

\subsubsection{Examples} \label{sec:examples}
The following two examples demonstrate how the appropriate choice of basis is essential for the statement of \cref{cyclicity}.
\begin{example}[Two variables with the same action] \label{ex:two-same-action}
    Consider the case $d=2$. 
    Let $\phi:\ZZ_p\to\Aut(H)$ be an automorphism, and suppose that both $\gamma_1$ and $\gamma_2$ act on $H$ via $\phi$, that is, $\gamma_i\cdot h=\phi(\gamma_i)(h)\cdot \gamma_i$ for $i=1,2$ and $h\in H$. 
    If the action is non-trivial, then $w_{\chi,1}=w_{\chi,2}\ne1$.
    If the basis were to remain unchanged, we would obtain $\tau_1=\tau_2$. An appropriate change of basis, on the other hand, would be replacing $(\gamma_1,\gamma_2)$ with $(\gamma_1\gamma_2^{-1}, \gamma_2)$: then $w_{\chi,1}=1$, so $\tau_1=1$.
\end{example}

\begin{example}[Two variables with different actions] \label{ex:two-different-actions}
    Let $p\colonequals3$, $F\colonequals \QQ_3$, and $d\colonequals2$.
    Let $H\colonequals C_9\times C_9=\langle x_1\rangle\times \langle x_2\rangle$ be the direct product of two copies of the cyclic group of order $9$.
    Let $G\colonequals (C_3\times C_3)\ltimes H=(\langle\overline{\gamma_1}\rangle \times \langle\overline{\gamma_2}\rangle)\ltimes H$, with conjugation by $\overline{\gamma_i}$ acting as $x_i\mapsto x_i^4$ and $x_{2-i}\mapsto x_{2-i}$, where $i=1,2$. In other words, $G\simeq 3_-^{1+2}\times 3_-^{1+2}$ is the direct product of the non-Heisenberg extraspecial group of order $27$ with itself. Letting $\G$ act on $H$ via the projection $\Gamma=\langle\gamma_1\rangle\times \langle\gamma_2\rangle\twoheadrightarrow G$, $\gamma_i\mapsto \overline{\gamma_i}$, this defines a semidirect product $\G\colonequals\Gamma\ltimes H$.
    
    Fix a primitive 9th root of unity $\zeta_9\in\QQ_3^\al$, and let $\eta\in\Irr(H)$ be the linear character given by $\eta(x_1^ix_2^j)\colonequals\zeta_9^{i+j}$. Inside $G$, the character $\eta$ has stabiliser $H$, and therefore $w_{\chi,i}=\#\langle\overline{\gamma_i}\rangle=3$.
    Define $\chi\colonequals\ind^\G_H\eta$ and $\chi_i\colonequals\ind^{\overline{\langle\gamma_i\rangle}\ltimes H}_H\eta$. By an application of \cite[Theorem~11.5(ii)]{CR}, we see that $\chi\in\Irr(\G)$ and $\chi_i\in\Irr(\overline{\langle\gamma_i\rangle}\ltimes H)$.
    A routine calculation shows that $F(\eta)=F_{\chi_1}=F_{\chi_2}=\QQ_3(\zeta_9)$, so $v_{\chi,i}=w_{\chi,i}$ and $\tau_i=1$.
    Meanwhile, $F_\chi=\QQ_3(\zeta_3)$, so $\langle\tau_1\rangle\times\langle\tau_2\rangle=1\lneq \Gal(F(\eta)/F_\chi)$.

    There is an isomorphism $G\simeq(\langle\overline{\gamma_1\gamma_2}\rangle\times\langle\overline{\gamma_2}\rangle)\ltimes H$, corresponding to choosing $(\gamma_1\gamma_2,\gamma_2)$ as a basis of $\Gamma$. Now for $\chi_i$ we have $F_{\chi_1}=\QQ_3(\zeta_3)$ and $F_{\chi_2}=\QQ_3(\zeta_9)$, so $v_{\chi,1}=1$ and $v_{\chi,2}=3$. Therefore with this choice of basis, we have $\langle\tau_1\rangle\times\langle\tau_2\rangle= \Gal(F(\eta)/F_\chi)$.
\end{example}

\begin{remark}
    In addition to \cref{cyclicity}, it is worth noting that the group $\Gal(F(\eta)/F_\chi)$ is `often' cyclic, that is, all but one $\tau_i$ are trivial. Indeed, since $F(\eta)$ is contained in the cyclotomic extension $F(\mu_{\#H})$ of $F$, the $p$-group $\Gal(F(\eta)/F_\chi)$ is isomorphic to a subquotient of $\Gal(\QQ_p(\mu_{\#H})/\QQ_p)$. If $p$ is an odd prime and $\#H=p^e \cdot n$ with $p\nmid n$, then the $p$-part of the latter abelian group is
    \[\Gal\big(\QQ_p(\mu_{\#H})/\QQ_p\big) [p^\infty] \simeq \ZZ/p^{e-1}\ZZ \times \prod_{\substack{q\mid n \text{ prime} \\ q\equiv 1\bmod p}} \big(\ZZ/(q-1)\ZZ\big)[p^\infty].\]
    In particular, if there are no such primes $q$ (or in other words, if $p\nmid\phi(n)$, where $\phi$ is Euler's phi function), then $\Gal(\QQ_p(\mu_{\#H})/\QQ_p) [p^\infty]$ is cyclic, and therefore so is $\Gal(F(\eta)/F_\chi)$.
\end{remark}
We provide an example in which $\Gal(F(\eta)/F_\chi)$ is non-cyclic. This shows that more than one factor in the product in \cref{cyclicity} may indeed be nontrivial.
\begin{example}[Non-cyclic Galois group]
    Let $p\colonequals 3$, $F\colonequals \QQ_3$, and $d\colonequals 2$.
    Let $H\colonequals C_7\times C_9=\langle x\rangle \times \langle y\rangle$ and $G\colonequals (C_3\times C_3)\ltimes H= (\langle \overline{\gamma_1}\rangle \times \langle \overline{\gamma_2}\rangle)\ltimes H$, where $\overline{\gamma_1}$ resp. $\overline{\gamma_2}$ act as $(x,y)\mapsto (x^2,y)$ resp. $(x,y)\mapsto(x,y^4)$. This defines a semidirect product $\G=\Gamma\ltimes H$.
    
    Let $\zeta_7$ resp. $\zeta_9$ be primitive 7th resp. 9th roots of unity. Consider the linear character $\eta\in\Irr(H)$ given by $\eta(x^iy^j)\colonequals\zeta_7^i\zeta_9^j$. As in \cref{ex:two-different-actions}, induction defines irreducible characters $\chi_i\in\Irr(\overline{\langle\gamma_i\rangle})$ for $i=1,2$, and $\chi\in\Irr(\G)$.
       
    We find that $F_{\chi_1}=\QQ_3(\xi,\zeta_9)$, $F_{\chi_2}=\QQ_3(\zeta_7,\zeta_3)$ and $F_\chi=\QQ_3(\xi,\zeta_3)$ where $\xi\colonequals\zeta_7+\zeta_7^2+\zeta_7^4$. Note that $\xi^2+\xi+2=0$, so $\QQ_3(\xi)=\QQ_3(\sqrt{-7})$.
    The fields $F_{\chi_1}$ and $F_{\chi_2}$ are both cubic extensions of $F_\chi$ inside $F(\eta)=\QQ_3(\zeta_7,\zeta_9)$. Since neither of the intermediate fields $F_{\chi_1}$ and $F_{\chi_2}$ contains the other, the extension $F(\eta)/F_\chi$ is non-cyclic.
    \[\begin{tikzcd}[ampersand replacement=\&]
    	\& {F(\eta)=\QQ_3(\zeta_7,\zeta_9)} \\
    	{F_{\chi_1}=\QQ_3(\sqrt{-7},\zeta_9)} \&\& {F_{\chi_2}=\QQ_3(\zeta_7,\zeta_3)} \\
    	\& {F_\chi=\QQ_3(\sqrt{-7},\zeta_3)}
    	\arrow["3", "\langle\tau_1\rangle"', no head, from=1-2, to=2-1]
    	\arrow["3"', "\langle\tau_2\rangle", no head, from=1-2, to=2-3]
    	\arrow["3", no head, from=2-1, to=3-2]
    	\arrow["3"', no head, from=2-3, to=3-2]
    \end{tikzcd}\]
    We have $w_{\chi,1}=w_{\chi,2}=3$, and therefore $v_{\chi,1}=v_{\chi,2}=1$.
\end{example}

\subsection{Idempotents}
In this subsection, it will sometimes be convenient to work with coefficients in a fixed algebraic closure $F^\al$ of $F$, as Ritter--Weiss did in \cite{TEIT-II}. As pointed out by Nickel \cite{NickelConductor}, it is in fact not necessary to go all the way to the algebraic closure: instead, one can fix a sufficiently large field extension of $F$ over which all characters involved have realisations, and use this as the field of coefficients.

Let $\chi\in\Irr(\G)$ be as in \cref{sec:invariants}, and let $\eta\mid\res^\G_H\chi$ be an irreducible constituent of its restriction. Consider the following group ring elements:
\begin{align*}
	e(\eta) &\colonequals \frac{\eta(1)}{\#H} \sum_{h\in H} \eta(h^{-1}) h \in F(\eta)[H], &
	e_\chi &\colonequals \sum_{g\in\G/\G_\eta} e({}^g\eta) \in F(\eta)[H], \\
	\epsilon(\eta) &\colonequals \sum_{\sigma\in\Gal(F(\eta)/F)} e({}^\sigma\eta) \in F[H], &
	\epsilon_\chi &\colonequals \sum_{\sigma\in\Gal(F_{\chi}/F)} \sigma(e_\chi) \in F[H].
\end{align*}
It follows from \eqref{eq:res-z-eta} and \cref{zchi=1} that
\[e_\chi = \frac{\chi(1)}{|H| \cdot w_\chi} \sum_{h\in H} \chi(h^{-1}) h,\]
which shows that $e_\chi\in F_{\chi}[H]$.
Furthermore, \cref{A1.1}.\ref{item:A1.1-etas} shows that
\begin{align*}
	\epsilon_\chi &= \sum_{\mathbf 0\le \ii< \frac{\w_{\chi}}{\vv_{\chi}}} \epsilon\left({}^{\ggamma^\ii}\eta\right).
\end{align*}

\subsubsection{Basic properties}
It is easily seen that each of these four elements is an idempotent. It is well known that $e(\eta)$ is a central primitive idempotent in $F(\eta)[H]$, and so the Galois-equivariant version $\epsilon(\eta)$ is also a central primitive idempotent in $F[H]$. Note the following:
\begin{lemma}[Orthogonality] \label{epsilon-orthogonality}
    For $\mathbf 0\le\ii,\jj<\w_\chi$, we have 
    \begin{align*}
        e\left({}^{\ggamma^\ii}\eta\right)e\left({}^{\ggamma^\jj}\eta\right)&=\begin{cases}
        e\left({}^{\ggamma^\ii}\eta\right) & \text{if }\ii=\jj, \\
        0 & \text{otherwise,}
        \end{cases}\\
        \epsilon\left({}^{\ggamma^\ii}\eta\right)\epsilon\left({}^{\ggamma^\jj}\eta\right)&=\begin{cases}
        \epsilon\left({}^{\ggamma^\ii}\eta\right) & \text{if }\ii\equiv\jj\pmod{\vv_\chi}, \\
        0 & \text{otherwise.}
        \end{cases}
    \end{align*}
    where congruence is understood coordinate-wise. \qed
\end{lemma}
Since $H$ is a normal subgroup of $\G$, the element $e_\chi$ is a central idempotent in $\Q^{\al}(\G)$. Consequently, the Galois-equivariant element $\epsilon_\chi$ is also a central idempotent in $\Q(\G)$.

The elements $e(\eta)$ and $\epsilon(\eta)$ are (not necessarily central) idempotents in $\Q^F(\G)$, and they are indecomposable by \cref{zchi=1}. Indeed, the module-theoretic version of Clifford theory \cite[Theorem~11.1.iii]{CR} tells us the following. Let $V_\chi$ be an $F^\al$-valued irreducible representation of $\G$ affording $\chi$: this decomposes as
\begin{equation} \label{eq:Vchi-decomposition}
    V_\chi=\bigoplus_{\eta\mid\res^\G_H\chi} e(\eta)V_\chi,
\end{equation}
where $\eta$ runs over the irreducible constituents of $\res^\G_H\chi$, and $V_\eta$ are the affording irreducible representations.
The resulting representations $e(\eta)V_\chi$ of $\G$ are also irreducible as $H$-representations by \cref{zchi=1}, so Schur's lemma provides isomorphisms
\begin{equation} \label{eq:Schur-lemma}
    \End_{F^\al} e(\eta) V_\chi \xrightarrow{\sim} F^\al[H]e(\eta).
\end{equation}

\subsubsection{Primitivity of $e_\chi$}
We adapt Ritter--Weiss's results to our setting, beginning with \cite[Proposition~5(2)]{TEIT-II}. 
\begin{proposition}[Higher Ritter--Weiss elements] \label{higher-RW-element}
    For each $1\le i\le d$, there is a unique element $\gamma_{\chi,i}\in\Q^F(\G)e_\chi$ such that left multiplication by $\gamma_{\chi,i}$ acts trivially on $V_\chi$, and $\gamma_{\chi,i}=g_{\chi,i} c_{\chi,i}$ for some lift $g_{\chi,i}\in\G$ of $\gamma_i^{w_{\chi,i}}\in\Gamma$ and $c_{\chi,i}\in(F^\al[H]e_\chi)^\times$.

    The following properties are satisfied:
    \begin{enumerate}
        \item For $g_{\chi,i}$ and $c_{\chi,i}$ as above, $g_{\chi,i} c_{\chi,i}=c_{\chi,i} g_{\chi,i}$ holds.
        \item The element $\gamma_{\chi,i}$ is central: $\gamma_{\chi,i}\in\cent(\Q^\al(\G)e_\chi)$.
        \item Some $p$-power of $\gamma_{\chi,i}$ is contained in $\G e_\chi$.
    \end{enumerate}
\end{proposition}
\begin{proof}
    We construct $\gamma_{\chi,i}$. Let $g_{\chi,i}$ be any lift of $\gamma_i^{w_{\chi,i}}$ under $\G\twoheadrightarrow\Gamma$. Since $g_{\chi,i}\in\G_\eta$, left multiplication by $g_{\chi,i}^{-1}$ acts component-wise on the right hand side of \eqref{eq:Vchi-decomposition}, so $g_{\chi,i}^{-1}$ defines an $F^\al$-endomorphism of the subspace $e(\eta)V_\chi$. Let $c_i(\eta)\in F^\al[H] e(\eta)$ be element corresponding to this endomorphism under \eqref{eq:Schur-lemma}. Define $c_{\chi,i}\colonequals \sum_{\eta\mid\res^\G_H\chi} c_i(\eta)\in F^\al[H] e_\chi$. Note that $c_{\chi,i}$ is invertible because each of the components $c(\eta)$ are on the respective subspaces $e(\eta)V_\chi$. Now define $\gamma_{\chi,i}\colonequals g_{\chi,i} c_{\chi,i}$. By construction, left multiplication by $\gamma_{\chi,i}$ is trivial on each irreducible component $e(\eta) V_\chi$, so it is trivial on the entire $V_\chi$, as claimed.

    The proof of uniqueness and the properties (1)--(3) carries over verbatim.
\end{proof}

Let $\Gamma_\chi$ denote the subgroup of $(\Q^\al(\G)e_\chi)^\times$ topologically generated by the elements constructed in \cref{higher-RW-element}:
\[\Gamma_\chi\colonequals\overline{\langle\gamma_{\chi,1},\ldots,\gamma_{\chi,d}\rangle} \subseteq (\Q^\al(\G)e_\chi)^\times.\]
Then $\Gamma_\chi\simeq \ZZ_p^d$. 
The next two statements generalise Ritter--Weiss's \cite{TEIT-II} Proposition~6 and its Corollary.
\begin{lemma} \label{QGechi-decomposition}
    There is a decomposition
    \[\Q^F(\G)e_\chi = \bigoplus_{\mathbf 0\le\ii<\w_\chi} \Q^F(\Gamma_\chi)[H]e_\chi \ggamma^\ii.\]
\end{lemma}
\begin{proof}
    The proof follows \cite[Proposition~6]{TEIT-II}.
    The fact that every element in $\Q^F(\G)e_\chi$ can be written as the sum of elements on the right hand side is a consequence of \cref{higher-RW-element}. The sum can be seen to be direct by using orthogonality of the idempotents $e(\eta)$ as follows. Let $x=\sum_{\mathbf 0\le\ii<\w_\chi} x_\ii\ggamma^\ii=0$ with $x_\ii\in\Q^F(\Gamma_\chi)[H]e_\chi\ggamma^\ii$.
    For $0<\jj,\kk\le\w_\chi$, the orthogonality property (\cref{epsilon-orthogonality}) implies:
    \begin{align*}
        0&=e\left({}^{\ggamma^\jj}\eta\right) x e\left({}^{\ggamma^{\kk}}\eta\right) = \sum_{\substack{\mathbf 0\le\ii<\w_\chi}} e\left({}^{\ggamma^\jj}\eta\right) x_\ii e\left({}^{\ggamma^{\kk-\ii}}\eta\right)\ggamma^\ii = e\left({}^{\ggamma^\jj}\eta\right) x_{\kk-\jj} \ggamma^{\kk-\jj}, 
    \end{align*}
    where $\kk-\jj$ is understood modulo $\w_\chi$. Hence $x_\ii=0$ for all $0<\ii<\w_\chi$.
\end{proof}

\begin{corollary} \label{e-chi}
    \begin{enumerate}[label=(\roman*), ref=(\roman*)]
        \item \label{item:RW-centre-iso} There is an isomorphism $\Q^F(\Gamma_\chi)\xrightarrow{\sim} \cent(\Q^F(\G)e_\chi)$.
        \item \label{item:e-chi-primitive} For all $\chi\in\Irr(\G)$, the element $e_\chi$ is a primitive central idempotent in $\Q^\al(\G)$. Conversely, every primitive central idempotent of $\Q^\al(\G)$ is of the form $e_\chi$ for some $\chi\in\Irr(\G)$.
    \end{enumerate}     
\end{corollary}
\begin{proof}
    Suppose that $x\in\cent(\Q^F(\G) e_\chi)$, and write $x=\sum_{\mathbf 0\le\ii<\w_\chi} x_\ii\ggamma^\ii$ under the decomposition of \cref{QGechi-decomposition}. Let $\mathbf 0\le\jj<\w_\chi$. By \cref{epsilon-orthogonality} and by centrality of $x$, we have
    \[\sum_{\mathbf 0\le\ii<\w_\chi} e\left({}^{\ggamma^\jj}\eta\right) x_\ii \ggamma^\ii = \sum_{\substack{\mathbf 0\le\ii<\w_\chi \\ \ii\equiv 0\pod{\vv_\chi}}} e\left({}^{\ggamma^\jj}\eta\right) x_\ii \ggamma^\ii = \sum_{\substack{\mathbf 0\le\ii<\w_\chi \\ \ii\equiv 0\pod{\vv_\chi}}} x_\ii e\left({}^{\ggamma^{\jj-\ii}}\eta\right) \ggamma^\ii \]
    It follows that
    \begin{align*}
        e\left({}^{\ggamma^\jj}\eta\right) x_\ii&=0 & \text{ whenever $\ii\not\equiv0\pmod{\vv_\chi}$,}\\
        e\left({}^{\ggamma^\jj}\eta\right) x_\ii&= x_\ii e\left({}^{\ggamma^{\jj-\ii}}\eta\right)= e\left({}^{\ggamma^\jj}\eta\right) x_\ii e\left({}^{\ggamma^{\jj-\ii}}\eta\right) & \text{ whenever $\ii\equiv0\pmod{\vv_\chi}$.}
    \end{align*}
    By orthogonality, this means that $e\left({}^{\ggamma^\jj}\eta\right) x_\ii=0$ unless $\ii=\mathbf 0$. Considering all $\mathbf 0\le\jj<\w_\chi$, we get $x_\ii=\sum_{\mathbf 0\le\jj<\w_\chi} e\left({}^{\ggamma^\jj}\eta\right) x_\ii=0$, so $x=x_{\mathbf 0}$ is central. This must commute with each $\gamma_i$, and \ref{item:RW-centre-iso} follows.

    Now we know that $\cent(\Q^F(\G)e_\chi)$ is a field (rather than a direct sum of fields), which shows that $e_\chi$ is primitive. The rest of \ref{item:e-chi-primitive} is proven in the same way as in \cite[p.~556, Corollary~(1)]{TEIT-II}.
\end{proof}

Finally, we record how $\gamma_{\chi,i}$ and $e_\chi$ change under twisting $\chi$ by a type~W character, that is, under replacing $\chi$ with $\chi\otimes\rho$, where $\rho$ is a character of $\chi$ such that $\res^\G_H\rho=\mathbbm 1_H$ is trivial.

\begin{lemma} \label{type-W}
    Let $\chi\in\Irr(\G)$.
    \begin{enumerate}[label=(\roman*), ref=(\roman*)]
        \item \label{item:type-W-i} If $\rho$ is a type~W character, then $\gamma_{\chi\otimes\rho,i}=\gamma_{\chi,i}\cdot \rho(\gamma_i^{w_{\chi,i}})$ for all $1\le i\le d$.
        \item \label{item:type-W-ii} Let $\chi'\in\Irr(\G)$. Then $e_\chi=e_{\chi'}$ if and only if $\chi'=\chi\otimes\rho$ for some type~W character $\rho$.
    \end{enumerate}
\end{lemma}
\begin{proof}
    The proof of \ref{item:type-W-i} is completely analogous to \cite[Proposition~5(3)]{TEIT-II}. Indeed, it is clear from the definitions that $e_{\chi\otimes\rho}=e_\chi$ and $w_{\chi\otimes\rho}=w_\chi$. Moreover, if $V_\chi$ resp. $V_\rho$ are the representations affording $\chi$ resp. $\rho$, then the representation affording $\chi\otimes \rho$ is $V_{\chi\otimes\rho}=V_\chi\otimes_{\QQ_p^\al}V_\rho$. Therefore in the notation of \cref{higher-RW-element}, we have $g_{\chi\otimes\rho_i}=g_{\chi_i}$. Then the same computation as in loc.cit. shows that $\gamma_{\chi,i}$ acts on $V_{\chi\otimes\rho}$ as $\rho(\gamma_i^{w_{\chi,i}})$, so \ref{item:type-W-i} follows from the uniqueness property of $\gamma_{\chi\otimes\rho,i}$.

    Assertion~\ref{item:type-W-ii} is an analogue of \cite[p.~556, Corollary~(2)]{TEIT-II}.
    For the non-trivial part of the statement, suppose that $e_\chi=e_{\chi'}$. Using the decomposition \eqref{eq:res-z-eta}, the triviality of Clifford multipliers (\cref{zchi=1}), and the fact that the idempotents $e({}^g\eta)$ are orthogonal, it follows that $\res^\G_H \chi=\res^\G_H\chi'$.
    Let $\eta\mid\res^\G_H\chi$ be an irreducible constituent. By Clifford theory \cite[Proposition~11.4(ii)]{CR}, we have $\chi=\ind^\G_{\G_\eta}\psi$ and $\chi'=\ind^\G_{\G_\eta}\psi'$ for some $\psi,\psi'\in\Irr(\G_\eta)$. By \cref{zchi=1}, we have $\res^{\G_\eta}_H\psi=\eta=\res^{\G_\eta}_H\psi'$. Applying \cite[Corollary~11.7]{CR}, we find that $\psi'=\psi\otimes\kappa$ for some linear character $\kappa\in\Irr(\G_\eta/H)$. Since $\G/H$ is abelian, $\kappa$ is $\G/H$-invariant, and therefore admits an extension to $\G/H$ by a successive application of \cite[Corollary~11.47]{CR} or \cite[Satz~V.17.12(a)]{Huppert}. Inflation from $\G/H$ to $\G$ then defines a type~W character $\rho$ such that $\chi'=\chi\otimes\rho$.
\end{proof}

\begin{remark}
    For a finite group $G$ and a normal subgroup $N$ containing the commutator subgroup, the relationship between characters $\chi,\chi'\in\Irr(G)$ whose restriction contains $\eta\in\Irr(N)$ is explored in \cite[\S2.2]{EllerbrockNickelStickelberger}. In this generality, Clifford multipliers need not be trivial. Note that in our setting, the commutator subgroup of $\G$ is contained in $H$.
\end{remark}

\subsection{The Iwasawa algebra as a crossed product} \label{sec:swan}
We recall the notion of crossed product rings, see e.g. \cite[Chapter~1, \S5.8]{McConnellRobson}. For a ring $R$ and a finite group $G$, a crossed product ring $R*G$ is a ring containing $R$ as a subring, together with an injective group homomorphism $\overline{\hspace{1mm}\cdot\hspace{1mm}}: G\to (R*G)^\times$ such that $R*G$ is a free right $R$-module with basis $\overline G$, and $\overline g R=R\overline g$ as well as $\overline g\cdot \overline{g'} R=\overline{gg'}R$ hold for all $g,g'\in G$.

Using the decomposition \eqref{eq:Lambda-G-decomp} of $\Lambda^{\OO_F}(\G)$ over $\Gamma_0$, the results of Nickel \cite[\S2.5]{swan} on the Cartan--Brauer theory of Iwasawa algebras readily generalise from dimension $1$ to $d$.
Indeed, the decomposition \eqref{eq:Lambda-G-decomp} together with \cite[\S2.3]{ArdakovBrown-survey} shows that $\Lambda^{\OO_F}(\G)$ is a crossed product ring:
\begin{equation}
    \Lambda^{\OO_F}(\G) \simeq \Lambda^{\OO_F}(\Gamma_0)*(\G/\Gamma_0).
\end{equation}
A $d$-dimensional version of Nickel's analogue of Swan's theorem \cite[Corollary~2.12]{swan}, and surjectivity of the connecting homomorphisms in the exact sequence of relative $K$-theory \cite[Corollaries~2.14--15]{swan} follow immediately.

\section{Multivariate skew power series rings over maximal orders of skew fields} \label{sec:skew-power-series}
We generalise the results of \cite[\S3]{W} to skew power series rings in multiple variables. 

\subsection{Definitions} \label{sec:skew-power-series-def}
Let $K/k$ be a finite Galois $p$-extension of local fields with Galois group $\Gal(K/k)=\prod_{i=1}^d \langle\tau_i\rangle$, where each $\tau_i\in\Gal(K/k)$ is a Galois automorphism of order $m_i$ (necessarily a $p$-power). Let $D$ be a skew field with centre $K$ and finite index $s$, and assume that $s\mid q_k-1$, where $q_k$ is the order of the residue field of $k$. For each $i=1,\ldots,d$, the automorphism $\tau_i\in\Gal(K/k)$ admits a unique extension to a $k$-automorphism of $D$ of the same order, as constructed in \cite[Proposition~2.7]{W}. By abuse of notation, from now on we will also write $\tau_i$ for these automorphisms. Let $\OO_K$ denote the ring of integers in $K$, and let $\OO_D$ be the unique maximal $\OO_K$-order in $D$.

Consider the multivariate skew power series ring 
\[\mathfrak O\colonequals \OO_D[[\X; \ttau,\ddelta]]\colonequals\OO_D[[X_1,\ldots,X_d;\, \tau_1,\ldots,\tau_d,\, \tau_1-1,\ldots,\tau_d-1]].\] 
The underlying abelian group of $\mathfrak O$ is the multivariate power series ring $\OO_D[[X_1,\ldots,X_d]]$. The multiplicative structure is given as follows: the variables $X_i$ commute with each other, and $X_i t = \tau_i(t) X_i+\tau_i(t)-t$ for all coefficients $t\in\OO_D$. One needs to check that this gives rise to a well-defined multiplication: this is done in the same way as in the $1$-dimensional case \cite[\S3.1]{W}, using the work of Schneider and Venjakob \cite{SchneiderVenjakob}.

The multivariate skew power series ring $\mathfrak O$ can be identified with the iterated skew power series ring
\begin{equation} \label{eq:iterated-skew-power-series}
    \mathfrak O\simeq \OO_D[[X_1;\tau_1,\tau_1-1]][[X_2;\tau_2,\tau_2-1]]\cdots[[X_d;\tau_d,\tau_d-1]]
\end{equation}
by letting the automorphisms $\tau_i$ act trivially on the variables $X_j$ for $1\le j<i\le d$.

\subsection{Centre, index, maximal orders}
\begin{proposition} \label{skew-power-series-centre}
    The centre of $\mathfrak O=\OO_D[[\X;\ttau,\ddelta]]$ is
    \[\cent\left(\OO_D[[\X;\ttau,\ddelta]]\right) = \OO_k[[T_1,\ldots,T_d]],\]
    where $T_i\colonequals (1+X_i)^{m_i}-1$.
\end{proposition}
\begin{proof}
    Consider the subrings $\mathfrak O_i\colonequals\OO_D[[X_i;\tau_i,\tau_i-1]]\subseteq \mathfrak O$. Note that the subring of $\mathfrak O$ generated by all elements of the subrings $\mathfrak O_1,\ldots,\mathfrak O_d$ is $\mathfrak O$ itself. Writing $C_{\mathfrak O}(\mathfrak O_i)$ for the centraliser of $\mathfrak O_i$ in $\mathfrak O$, it follows that $\cent(\mathfrak O)=\bigcap_{i=1}^d C_{\mathfrak O}(\mathfrak O_i)$.
    
    Let $1\le i\le d$. We have $\cent(\mathfrak O_i)=\OO_{K^{\langle \tau_i\rangle}}[[T_i]]$ by \cite[Corollary~3.3]{W}. Since conjugation by each $X_j$ acts as $\tau_j$ on $K^{\langle\tau_i\rangle}$, and since the variables $X_j$ commute with each other, it follows that $\mathfrak O_i$ has centraliser $C_{\mathfrak O}(\mathfrak O_i)=\OO_{K^{\langle\tau_i\rangle}}[[T_1,\ldots,T_d]]$. Taking intersection over $1\le i\le d$ and using $k=\bigcap_{i=1}^d K^{\langle \tau_i\rangle}$, the claim follows.
\end{proof}

From now on, let $\mathfrak D\colonequals\Quot(\OO_D[[\X;\ttau,\ddelta]])$ denote the total ring of quotients of $\mathfrak O$. Let $\OO_{\cent(\mathfrak D)}\colonequals \OO_k[[\T]]\colonequals\OO_k[[T_1,\ldots,T_d]]$.

\begin{corollary} \label{skew-power-series-index}
    The total ring of quotients $\mathfrak D=\Quot(\OO_D[[\X;\ttau,\ddelta]])$ is a skew field of index $[K:k]s$.
\end{corollary}
\begin{proof}
    Using the identification \eqref{eq:iterated-skew-power-series}, an iterated application of \cite[Corollary~2.10(i)]{Venjakob-WPT} shows that $\mathfrak O$ is a domain, so $\mathfrak D$ is a skew field. Then $\mathfrak D$ is a vector space over its centre $\Frac(\OO_k[[T_1,\ldots,T_d]])$ by \cref{skew-power-series-centre}, and 
    \[\dim_{\cent(\mathfrak D)}\mathfrak D=\dim_k (D) \cdot \prod_{i=1}^d m_i=\dim_K (D) \cdot [K:k] \cdot \prod_{i=1}^d m_i = s^2 \cdot [K:k]^2. \qedhere\]
\end{proof}

By work of Hasse \cite[\S14]{MO}, the skew field $D$ can be written as a cyclic algebra 
\begin{equation} \label{eq:Hasse-description}
    D=(K(\omega)/K,\sigma,\pi_K)=\bigoplus_{i=0}^{s-1} K(\omega) \pi_D^i
\end{equation}
with multiplication given by $\pi_D \omega=\sigma(\omega)\pi_D$. Here $\pi_K$ is a uniformiser of $K$, $\pi_D$ is an $s$th root of $\pi_K$, $\omega$ is a root of unity of order $q^s-1$ with $q$ being the order of the residue field of $K$, and $\sigma(\omega)=\omega^{q^r}$, where $r/s$ is the Hasse invariant of $D$.

\begin{proposition} \label{splitting}
    The field $\mathfrak E\colonequals\Frac(\OO_{K(\omega)}[[T_1,\ldots,T_d]])$ is a maximal subfield in $\mathfrak D$, and there is a splitting map $\Phi: \mathfrak D \hookrightarrow M_{[K:k]s}(\mathfrak E)$, which will be given below explicitly.
\end{proposition}
\begin{proof}
    The field $\mathfrak E$ is a maximal subfield of $\mathfrak D$ because it has degree $[K:k]s$ over $\cent(D)$. We now construct the splitting map $\Phi$. Let $\phi:D\to K(\omega)\otimes_K D\to M_s(K(\omega))$ be the composition of the map given by $t\mapsto 1\otimes t$ with Hasse's splitting isomorphism $K(\omega)\otimes_K D\to M_s(K(\omega))$ described in \cite[Theorem~14.6]{MO}. 
    Consider the subfield $\mathfrak L\colonequals \Frac(\OO_K[[T_1,\ldots,T_d]])=\Frac(\OO_K[[\mathbf T]])$ of $\mathfrak D$. Let $\hat\phi: \Quot(\OO_D[[\T]])\to \mathfrak E\otimes_{\mathfrak L} \Quot(\OO_D[[\T]])\to M_s(\mathfrak E)$ be the extension of $\phi$ defined by $\phi(T_i)\colonequals T_i\1_s$, where $\mathbf 1_s$ is the $s\times s$ identity matrix.

    Let the automorphisms $\tau_i$ act on $M_s(\mathfrak E)$ coefficient-wise (and as before, trivially on the variables $T_j$).
    For $g\in\Quot(\OO_D[[\T]])$, let $\Phi(g)$ be the block diagonal matrix with entries $\ttau^{\ll}\hat\phi(g)$ in its diagonal, where $\ll\in \prod_{i=1}^d ([0,m_i-1]\cap \ZZ)$, and where the tuples $\ll$ are ordered lexicographically:
    \begin{align*}
        \Phi(g)=\diag&\Big(\hat\phi(g), \tau_d\hat\phi(g),\tau_d^2\hat\phi(g),\ldots,\tau_d^{m_d-1}\hat\phi(g), \\
        &\tau_{d-1}\hat\phi(g), \tau_{d-1}\tau_d\hat\phi(g),\tau_{d-1}\tau_d^2\hat\phi(g),\ldots,\tau_{d-1}\tau_d^{m_d-1}\hat\phi(g),\\
        &\ldots,\\
        &\tau_1^{m_1-1}\dots\tau_{d-1}^{m_{d-1}-1}\hat\phi(g), \tau_1^{m_1-1}\dots\tau_{d-1}^{m_{d-1}-1}\tau_d\hat\phi(g),\ldots,\tau_1^{m_1-1}\dots\tau_{d-1}^{m_{d-1}-1}\tau_d^{m_d-1}\hat\phi(g)\Big).
    \end{align*}
    
    Let $\Phi(1+X_i)$ be the block diagonal matrix with $\prod_{j=1}^{i-1} m_j$ diagonal blocks, each of size $n_i\colonequals s\prod_{j=i+1}^{d} m_j$, of the shape
    \[\begin{psmallmatrix}
        &\1_{n_i} \\ && \ddots \\ &&& \1_{n_i} \\ (1+T_i)^{m_i}\1_{n_i}
    \end{psmallmatrix} \in M_{m_i n_i}(\mathfrak E).\]
    Recall that $[K:k]=\prod_{i=1}^d m_i$, so we have $\Phi(1+X_i)\in M_{[K:k]s}(\mathfrak E)$.
    By easy calculations, we have $\Phi(1+X_i)\Phi(g)=\tau_i\Phi(g)\Phi(1+X_i)$ for all $i=1,\ldots,d$ and $g\in\Quot(\OO_D[[\T]])$.
    This defines an $\mathfrak L$-algebra homomorphism $\Phi$ as in the statement. As $\mathfrak D$ is a skew field and $\Phi\ne 0$, injectivity is clear.
\end{proof}

For the description of maximal orders, recall the following result of Ramras.
\begin{proposition}[{\cite[Theorem~5.4]{Ramras}}] \label{Ramras}
    Let $R$ be a regular local ring of dimension $d$, let $\mathcal A$ be a finite dimensional central simple algebra over $\Frac(R)$, and let $\mathfrak M$ be an $R$-order in $\mathcal A$. 
    
    Suppose that $\mathfrak M$ has finite global dimension and that the quotient $\mathfrak M/\rad \mathfrak M$ of $\mathfrak M$ by its Jacobson radical is simple artinian. 
    Then $\mathfrak M$ is a maximal $R$-order in $\mathcal A$, its global dimension is $d$, and the maximal $R$-orders in $\mathcal A$ are precisely the $\mathcal A^\times$-conjugates of $\mathfrak M$.
\end{proposition}

Ramras proves this result for regular local rings of dimension $d=2$, but every step of the proof, including the references to \cite{AuslanderGoldman}, works in higher dimension as well.

\begin{proposition} Maximal orders in matrix rings over $\mathfrak D$ are described as follows.
    \begin{enumerate}[label=(\roman*), ref=(\roman*)]
        \item \label{item:maximal-orders-in-D} The ring $\mathfrak O$ is a maximal $\OO_{\cent(\mathfrak D)}$-order in $\mathfrak D$, and every maximal order is conjugate to $\mathfrak O$.
        \item \label{item:maximal-orders-in-GL} For all $n\ge1$ and $u\in\GL_n(\mathfrak D)$, the ring $u M_n(\mathfrak O)u^{-1}$ is a maximal $\OO_{\cent(\mathfrak D)}$-order in $M_n(\mathfrak D)$. Conversely, every maximal order is of this form.
    \end{enumerate}
\end{proposition}
\begin{proof}
    \ref{item:maximal-orders-in-D} Using the identification \eqref{eq:iterated-skew-power-series}, we can iteratively apply \cite[Corollary~2.10(iii)]{Venjakob-WPT}: this shows that $\mathfrak O$ is a maximal order. The conditions of Venjakob's result are verified in the same way as in \cite[Lemma~3.4]{W}.
    For the second assertion, note that $\OO_{\cent(\mathfrak D)}$ is a regular local ring of dimension $d$, so \cref{Ramras} is applicable.

    \ref{item:maximal-orders-in-GL} The first assertion is a special case of \cite[Theorem~8.7]{MO}, and the converse follows from \cref{Ramras}.
\end{proof}

\section{Skew fields cut out by idempotents of group rings and completed group rings} \label{sec:Wedderburn}
Let $F/\QQ_p$ be a finite extension. Then by Maschke's theorem, the group ring $F[H]$ is semisimple:
\[F[H]=\bigoplus_{\eta\in\Irr(H)/\sim_F} M_{n_\eta}(D_\eta),\]
where $\eta$ runs over equivalence classes of irreducible characters of $H$, with two characters $\eta$ and $\eta'$ being equivalent if there is $\sigma\in\Gal(F(\eta)/F)$ such that ${}^\sigma\eta=\eta'$. For each $\eta$, $D_\eta$ is a skew field with centre $F(\eta)$, and we let $s_\eta$ denote the Schur index. By a result of Witt, we have $s_\eta\mid p-1$ \cite[Satz~10]{Witt}.

The completed group algebra $\Q^F(\G)$ is semisimple by \cref{semisimplicity}. By \cref{e-chi}\ref{item:e-chi-primitive} and \cref{type-W}\ref{item:type-W-ii}, the Wedderburn decomposition is of the form
\[\Q^F(\G)=\bigoplus_{\chi\in\Irr(\G)/\sim_F} M_{n_\chi}(D_\chi),\]
where two characters $\chi,\chi'$ of $\G$ are equivalent if their restrictions to $H$ are equivalent in the sense of the definition above.

In this section, we describe the Wedderburn decomposition of $\Q^F(\G)$ in terms of the Wedderburn decomposition of the group ring $F[H]$.

\subsection{Automorphisms in \texorpdfstring{$\eta$}{η}-components}
This subsection is the $d$-dimensional generalisation of \cite[\S5]{W}.
The Wedderburn components of $F[H]$ are cut out by the idempotents $\epsilon(\eta)$ with $\eta\in\Irr(H)$:
\begin{equation} \label{eq:FH-Wedderburn}
    F[H]\epsilon(\eta)\xrightarrow{\sim} M_{n_\eta}(D_\eta).
\end{equation}
For the rest of this section, fix an irreducible character $\eta\in\Irr(H)$. Let $i\in\{1,\ldots,d\}$. Conjugation by $\gamma_i^{v_{\chi,i}}$ defines an automorphism of $F[H]\epsilon(\eta)$. Pushing forward along \eqref{eq:FH-Wedderburn}, this defines an automorphism $\Delta_i\in\Aut(M_{n_\eta}(D_\eta))=\GL_{n_\eta}(D_\eta)$. On the other hand, entry-wise application of $\tau_i$ is an automorphism of $M_{n_\eta}(D_\eta)$, and pulling this back along \eqref{eq:FH-Wedderburn} defines an automorphism $\delta_i\in\Aut(F[H]\epsilon(\eta))$.

\begin{proposition} \label{y-elements}
    \begin{enumerate}[label=(\roman*), ref=(\roman*)]
        \item For each $i=1,\ldots,d$, there exists $y_{i}\in F[H]\epsilon(\eta)^\times$ such that
            \[\forall x\in F[H]\epsilon(\eta): \delta_i\left(\gamma_i^{v_{\chi,i}} x \gamma_i^{-v_{\chi,i}}\right) = y_i x y_i^{-1}\]
        \item If $Y_{i}\in M_{n_\eta}(D_\eta)$ is the image of $y_i$ under \eqref{eq:FH-Wedderburn}, then
            \[\forall X\in M_{n_\eta}(D_\eta): \Delta_i X = \tau_i\left(Y_i X Y_i^{-1}\right).\]
        \item For $i,j\in\{1,\ldots,d\}$, $y_iy_j=y_jy_i$, $\delta_i\delta_j=\delta_j\delta_i$, and $\Delta_i\Delta_j=\Delta_j\Delta_i$.
    \end{enumerate}
\end{proposition}
\begin{proof}
    Conjugation by $\gamma_i^{v_{\chi,i}}$ and $\delta_i$ define the same automorphism on the centre $\cent(F[H]\epsilon(\eta))$: this is \cite[Proposition~5.3]{W}. (While the result there is stated for $\Q^F(\Gamma_0)[H]\epsilon(\eta)$ with $\Gamma_0$ one-dimensional, in fact, everything takes place over the group ring $F[H]$, as pointed out in \cite[Remark 5.2]{W}.) The Skolem--Noether theorem provides a unit $y_i$ (determined up to central units) satisfying the properties in the first assertion. The second assertion follows from the first one. The third assertion follows from the fact that conjugations by $\gamma_i$ and $\gamma_j$ commute.
\end{proof}

For $\ll=(\ell_1,\ldots,\ell_d)\in\ZZ^d$, define
\[a_\ll\colonequals\prod_{k=1}^d \delta_k(y_k) \delta_k^2(y_k)\dots \delta_k^{\ell_k}(y_k) \in F[H]\epsilon(\eta);\]
the product is well-defined by \cref{y-elements}(3). 
For $x\in \Q^F(\Gamma_0)[H]\epsilon(\eta)$, let 
\[\delta^\ll(x)\colonequals \delta_1^{\ell_1}\delta_2^{\ell_2}\dots\delta_d^{\ell_d}(x)\in \Q^F(\Gamma_0)[H]\epsilon(\eta).\]
Then \cref{y-elements}(1) implies
\begin{equation} \label{eq:conjugation-by-ggamma-ll}
    {}^{\ggamma^{\ll\vv_\chi}}x=a_\ll \delta^\ll(x) a_\ll^{-1},
\end{equation}
where $\ll\vv_\chi$ is the entry-wise product of $\ll$ and $\vv_\chi$.

Let $f_{\eta}^{(j)}\in M_{n_\eta}(D_\eta)$ denote the indecomposable idempotent corresponding under \eqref{eq:FH-Wedderburn} to the diagonal matrix with $1$ in the $j$th entry of the diagonal and zero everywhere else. Note that $\epsilon(\eta)=\sum_{j=1}^{n_\eta} f_\eta^{(j)}$ and $f_\eta^{(j)}F[H]f_\eta^{(j)}\simeq D_\eta$. We describe the multiplicative structure of the ring $f_\eta^{(j)} \Q^F(\G) f_\eta^{(j)}$. Let $x,x'\in\Q^F(\G)$: then the decomposition \eqref{eq:QG-pn0} allows us to write 
\[x=\sum_{\mathbf 0\le \ll<p^{\mathbf n_0}} x_\ll\ggamma^\ll\] 
with $x_\ll\in\Q^F(\Gamma_0)[H]$, and analogously for $x'$. By an application of \cref{epsilon-orthogonality} (see also \cite[Lemma~5.11]{W}), only terms with $\ll,\ll'\equiv0\pmod{\vv_\chi}$ are nonzero, and \eqref{eq:conjugation-by-ggamma-ll} then shows:
\begin{align} \label{eq:multiplication-rule}
    f_\eta^{(j)} x f_\eta^{(j)} \cdot f_\eta^{(j)} x' f_\eta^{(j)} &= \sum_{\substack{\mathbf 0\le \ll,\ll'<p^{\mathbf n_0} \\ \ll,\ll'\equiv 0\pmod{\vv_\chi}}} f_\eta^{(j)} x_\ll a_\ll f_\eta^{(j)} \cdot f_\eta^{(j)} \delta^\ll(x'_{\ll'}) \delta^\ll(a_{\ll'}) f_\eta^{(j)} \cdot a_{\ll+\ll'}^{-1} \ggamma^{\ll+\ll'}
\end{align}

Let $\gamma_{\eta,i}''\colonequals y_i^{-1} \gamma_i^{v_{\chi,i}} \in \Q(\G)\epsilon(\eta)$, and let $\Gamma_{\eta,i}\colonequals\overline{\langle\gamma_{\eta,i}\rangle}$ be the procyclic group generated by it.
It follows from \cref{y-elements}(3) that $\gamma_{\eta,i}''\gamma_{\eta,j}''=\gamma_{\eta,j}''\gamma_{\eta,i}''$ for $i,j\in\{1,\ldots,d\}$.
Therefore 
\[\ggamma_{\eta,\ll}''\colonequals \prod_{i=1}^d (\ggamma''_{\eta,i})^{\ll_i} \in \Q(\G)\epsilon(\eta)\] 
is well-defined.
Conjugation by $\gamma_{\eta,i}''$ acts as $\delta_i$ on $f_{\eta}^{(j)} \Q^F(\Gamma_0)[H] f_{\eta}^{(j)}$, and so conjugation by $\ggamma''_{\eta,\ll}$ acts as $\delta^\ll$.

\subsection{Skew fields in \texorpdfstring{$\chi$}{χ}-components --- totally ramified case}
The following results are the $d$-dimensional generalisations of \cite[\S6]{W}.
\begin{proposition} \label{skew-field-totally-ramified}
    Let $\chi\in\Irr(\G)$, and let $\eta\mid\res^\G_H\chi$ be an irreducible constituent. Suppose that $F(\eta)/F_\chi$ is totally ramified. Then $f_{\eta}^{(j)} \Q^F(\G) f_{\eta}^{(j)}$ is a skew field for all $1\le j\le n_\eta$.
\end{proposition}
\begin{proof}
	Let $\chi$, $\eta$ and $j$ be fixed as in the statement.
    In this proof, we write $f\colonequals f_{\eta}^{(j)}$ for brevity. Let $x'\in\Q^F(\G)$ such that $fx'f\ne0$: we will show that there exists an $x\in\Q^F(\G)$ such that $fxf\cdot fx'f=f$. Comparing the multiplication rule \eqref{eq:multiplication-rule} with the direct sum decomposition \eqref{eq:QG-pn0} then provides equations for $\mathbf 0\le\kk<p^{\mathbf n_0}$:
    \[\sum_{\substack{\mathbf 0\le \ll,\ll'<p^{\mathbf n_0} \\ \ll,\ll'\equiv 0\pmod{\vv_\chi} \\ \ll+\ll'\equiv \kk\pmod{p^{\mathbf n_0}}}} f x_\ll a_\ll f\cdot f \delta^\ll(x_{\ll'} a_{\ll'} f) \cdot a_{\ll+\ll'}^{-1} \ggamma_0^{\mathbf t}=\updelta_{0,\kk} f,\]
    where $\updelta_{0,\kk}$ is the Kronecker delta, and $\mathbf t\colonequals(\ll+\ll'-\kk)/p^{\mathbf n_0}\in\ZZ^d$, where operations are understood entry-wise. Then $a_{\ll+\ll'}^{-1}=a_{\kk+\mathbf t p^{\mathbf n_0}}^{-1}=\delta^\kk(a_{\mathbf tp^{\mathbf n_0}})^{-1} \cdot a_\kk^{-1}$: here $a_{\mathbf tp^{\mathbf n_0}}$ is central by \eqref{eq:conjugation-by-ggamma-ll}. Therefore the system of equations becomes
    \[\sum_{\substack{\mathbf 0\le \ll,\ll'<p^{\mathbf n_0} \\ \ll,\ll'\equiv 0\pmod{\vv_\chi} \\ \ll+\ll'\equiv \kk\pmod{p^{\mathbf n_0}}}} f x_\ll a_\ll f\cdot f \delta^\ll(x_{\ll'} a_{\ll'}) \delta^\kk(a_{\mathbf t p^{\mathbf n_0}}) f \cdot \ggamma_0^{\mathbf t}=\updelta_{0,\kk} f.\]
    Let $\tilde x_\ll$ resp. $\tilde x_{\ll'}'$ be the image of the first resp. second factor in the summation under the isomorphism $f \Q^F(\Gamma_0)[H] f\simeq \tD_\eta\colonequals D_\eta\otimes_{F(\eta)}\Q^{F(\eta)}(\Gamma_0)$. We obtain a system of linear equations 
    \[\sum_{\substack{\mathbf 0\le \ll,\ll'<p^{\mathbf n_0} \\ \ll,\ll'\equiv 0\pmod{\vv_\chi} \\ \ll+\ll'\equiv \kk\pmod{p^{\mathbf n_0}}}} \tilde x_\ll \tilde x'_{\ll'} \gamma_0^{\mathbf t}=\updelta_{0,\kk}\]
    over the skew field $\tD_\eta$ with variables $\tilde x_\ll$. The existence of a solution is equivalent to the existence of an $x$ as above (cf. \cite[Lemma~6.2]{W}).
    
    Define a ring 
    \[A\colonequals\bigoplus_{\mathbf 0\le \jj<\frac{p^{\mathbf n_0}}{\vv_\chi}} \tD_\eta \ggamma^{\jj \vv_\chi}\] 
    with multiplication rule $\gamma_i t=\tau_i(t)\gamma_i$ for $t\in\tD_\eta$ for $1\le i\le d$. This is the $d$-dimensional analogue of the ring defined before Remark~6.2 of \cite{W}. The elements $\ggamma^{\jj \vv_\chi}$ form a $\tD_\eta$-basis of $A$, and finding a left inverse to a nonzero element in $A$ reduces to the same system of linear equations as above, so $f \Q^F(\G) f$ is a skew field if and only if $A$ is. We now show this to be the case.

    For $i=1,\ldots,d$, consider the subrings 
    \[A_i\colonequals\bigoplus_{0\le j< \frac{p^{n_i}}{v_{\chi,i}}} \tD_\eta \gamma_i^j\subset A.\] 
    During the proof of \cite[Theorem~6.1]{W}, it is proven that each $A_i$ is a skew field: this is where the condition on total ramification is used.
    The skew field $A_i$ is a left $\tD_\eta$-module in the obvious way, and a right $\tD_\eta$-module via $\tau_i$. We have $\cent(A_i)=\cent(\tD_\eta)$, and $A_i\cap A_j=\tD_\eta$ whenever $i\ne j$. Moreover, $A=\bigotimes_{i=1}^d A_i$, where the tensor product is taken over $\tD_\eta$. It follows that $A$ is a central simple $\cent(\tD_\eta)$-algebra by \cite[Theorem~3.60(iv)]{CR}, so it is a matrix ring over some skew field. If $A$ were a matrix ring but not a skew field, it would have nilpotent elements, but it is clear from definition that $A$ contains no nilpotent elements. So $A$ is a skew field.
\end{proof}

\begin{corollary} \label{Wedderburn-totally-ramified}
    Suppose that $F(\eta)/F_\chi$ is totally ramified. Then the following hold:
    \begin{corollarylist}
        \item \label{item:W-t-r-i} $n_\chi=n_\eta v_\chi$ and $s_\chi=s_\eta w_\chi/v_\chi$.
        \item \label{item:W-t-r-ii} $\displaystyle D_\chi\simeq f_{\eta}^{(j)} \Q^F(\G) f_{\eta}^{(j)}\simeq \bigoplus_{\substack{\mathbf 0\le \ll<p^{\mathbf n_0} \\ \ll\equiv 0\pmod{\vv_\chi}}} \tD_\eta \ggamma''_{\eta,\ll} \simeq \Quot(\OO_{D_\eta}[[\X; \ttau, \ddelta]])$. \label{Wedderburn-totally-ramified-ii}
    \end{corollarylist}
\end{corollary}
\begin{proof}
    The first assertion follows by the same argument as in \cite[\S6.2]{W}. The first isomorphism of \ref{item:W-t-r-ii} follows from the fact that $f_{\eta}^{(j)} \mid \epsilon_\chi$. We obtain the second isomorphism by applying \eqref{eq:QG-pn0}. The third map is given by sending $\ggamma''_{\eta,\ll}\mapsto (1+\X)^\ll=\prod_{i=1}^d (1+X_i)^{\ll_i}$, and it is the identity on $D_\eta$. This already determines the image of $\ggamma_0$ by the same argument as in the proof of \cite[Proposition~6.9]{W}. It is easily seen to be a homomorphism, it is clearly surjective, and the dimensions of its domain resp. codomain over their respective centres are equal by \ref{item:W-t-r-i} resp. \cref{skew-power-series-index}.
\end{proof}

\subsection{Skew fields in \texorpdfstring{$\chi$}{χ}-components --- arbitrary ramification}

\subsubsection{Preliminaries on base fields}
Let $\chi\in\Irr(\G)$, and $\eta\mid\res^\G_H\chi$ be an irreducible constituent.
Let $W$ be the maximal unramified extension of $F_\chi$ in $F(\eta)$. Note that $W(\eta)=F(\eta)$ and $W_\chi=W$, so $W(\eta)/W_\chi$ is totally ramified, and therefore $\Q^W(\G)$ is described by \cref{Wedderburn-totally-ramified}. In this section, we extend these results to $F$. In doing so, it will be important to keep track of which base field each object is defined over, which we will denote by adding a subscript $F$ or $W$, e.g. $D_{F,\chi}$ is a skew field in the Wedderburn decomposition of $\Q^F(\G)$.

It is easily seen that for the group rings $F[H]$ and $W[H]$, we have $D_{F,\eta}=D_{W,\eta}$, $s_{F,\eta}=s_{W,\eta}$, and $n_{F,\eta}=n_{W,\eta}$ for all $\eta\in\Irr(H)$; see the beginning of \cite[\S3.1]{W} for an explanation. In these cases, we shall omit the base field from the notation.

\begin{lemma} \label{inertia-direct-product}
    The inertia group of $F(\eta)/F_\chi$ is the direct product of the inertia groups of the extensions  $F(\eta)/F(\eta)^{\langle\tau_i\rangle}$ for $i=1,\ldots,d$. If $f$ resp. $f_i$ denote the inertia degrees of $F(\eta)/F_\chi$ resp. $F(\eta)/F(\eta)^{\langle\tau_i\rangle}$, then $\prod_{i=1}^d f_i=f$ and $\vv_{W,\chi}=\vv_{F,\chi} \ff$, where $\ff=(f_i)$. In particular, $v_{W,\chi}=v_{F,\chi} f$.
\end{lemma}

\begin{proof}[Proof of \cref{inertia-direct-product}]
    Let $I$ denote the inertia group of $F(\eta)/F_\chi$, and let $I_i$ be the inertia groups of $F(\eta)/F(\eta)^{\langle\tau_i\rangle}$. Then $I_i=\langle\tau_i^{f_i}\rangle$. By basic properties of inertia groups \cite[Proposition~II.9.5]{NeukirchANT}, we have $I_i=I\cap \Gal(F(\eta)/F(\eta)^{\langle\tau_i\rangle})$, so there is an inclusion $\prod_{i=1}^d I_i\le I$. In particular, $\Pi(\ff)=\prod_{i=1}^d f_i\mid f$.

    We have ${}^{(\ttau_F)^\ff}\eta={\ggamma^{\vv_{F,\chi} \ff}}$ and $(\ttau_F)^\ff\in\prod_{i=1}^d I_i\le I$, so it follows from \cref{def:vv} that $\vv_{W,\chi} \mid \vv_{F,\chi}\ff$.
    Applying \cref{cyclicity} twice, we have
    \[\frac{w_\chi}{v_{F,\chi}}=[F(\eta):F_\chi]=[F(\eta):W]\cdot[W:F_\chi]=[W(\eta):W_\chi]\cdot f = \frac{w_\chi}{v_{W,\chi}} \cdot f,\]
    which shows that $v_{W,\chi}=v_{F,\chi} f$. Therefore 
    \[v_{W,\chi}=\Pi(\vv_{W,\chi}) \,\Big|\, \Pi(\vv_{F,\chi})\cdot\Pi(\ff)=v_{F,\chi}\cdot \Pi(\ff) \,\Big|\, v_{F,\chi} \cdot f,\]
    which forces $\Pi(\ff)=f$ and $\vv_{W,\chi} = \vv_{F,\chi}\ff$. Now $\prod_{i=1}^d I_i=I$ follows from $\Pi(\ff)=f$.
\end{proof}

\begin{remark}
    In general, in a compositum of local fields, the inertia group is not necessarily the direct product of inertia groups associated with the individual fields. This is due to Goursat's lemma, see \cite{Lahtonen} for a counterexample.
\end{remark}

\subsubsection{Indecomposable idempotents}
We generalise the results of \cite[\S3]{W2}.
As in \cref{A1.1}, there are irreducible characters $\eta_i$, $1\le i\le v_{F,\chi}$ such that
\[\res^\G_H\chi = \sum_{i=1}^{v_{F,\chi}} \sum_{\sigma\in\Gal(F(\eta)/F_\chi)} {}^\sigma\eta_i.\]
\begin{definition} \label{def:f}
    For $1\le i\le v_{F,\chi}$ and $1\le j\le n_\eta$, let 
    \[\displaystyle \f^{(j)}_{F,\eta_i}\colonequals\sum_{\psi\in\Gal(W/F)} \psi\left(f^{(j)}_{W,\eta_i}\right) \in \Q^F(\G)\epsilon_{F,\chi}.\]
\end{definition}

\begin{lemma}[Orthogonality] \label{f-orthogonalities}
    For $1\le i,i'\le v_{F,\chi}$, $1\le j,j'\le n_\eta$ and $\psi,\psi'\in\Gal(W/ F)$, 
    \[\psi\left(f^{(j)}_{W,\eta_i}\right) \cdot \psi'\left(f^{(j')}_{W,\eta_{i'}}\right)=\updelta_{\psi,\psi'} \updelta_{i,i'} \updelta_{j,j'} \psi\left( f^{(j)}_{W,\eta_i}\right).\]
    In particular, $\f^{(j)}_{F,\eta_i}$ is an idempotent.
\end{lemma}
\begin{proof}
    The proofs of Lemma~3.1 and Corollary~3.2 of \cite{W2} apply without modification.
\end{proof}

\begin{proposition} \label{skew-field-general}
    For $1\le i\le v_{F,\chi}$ and $1\le j\le n_\eta$, the ring $\f^{(j)}_{F,\eta_i}\Q^F(\G)\f^{(j)}_{F,\eta_i}$ is a skew field.
\end{proposition}
\begin{proof}
    The proof is completely analogous to \cite[Proposition~3.3]{W}. For $x'\in\Q^F(\G)$ such that $\f^{(j)}_{F,\eta_i} x' \f^{(j)}_{F,\eta_i}\ne 0$, we need to show the existence of an $x\in\Q^F(\G)$ such that 
    \[\f^{(j)}_{F,\eta_i} x\f^{(j)}_{F,\eta_i} \cdot \f^{(j)}_{F,\eta_i} x'\f^{(j)}_{F,\eta_i} =\f^{(j)}_{F,\eta_i}.\] 
    This leads to a system of linear equations, and orthogonality (\cref{f-orthogonalities}) decomposes these into systems of linear equations, one for each Galois automorphism $\psi$. It is sufficient to solve one of these, say, the one for the identity automorphism: twisting the solution by $\psi\in\Gal(W/F)$ then produces a solution for the $\psi$-system. The existence of such a solution was shown in \cref{skew-field-totally-ramified}.
\end{proof}

\subsubsection{The Wedderburn decomposition}

Now we can finish determining the $d$-dimensional Wedderburn decomposition under arbitrary ramification, thereby generalising Corollaries~3.5 and 4.3 of \cite{W2}.

\begin{corollary} \label{s-n-general}
    We have $n_{F,\chi}=n_\eta v_{F,\chi} $ and $s_{F,\chi}=s_\eta [F(\eta):F_\chi]$.
\end{corollary}
\begin{proof}
    On the one hand, $\epsilon_{F,\chi}$ is the sum of $n_{F,\chi}$ indecomposable idempotents. On the other hand, it follows from the definitions that there is a decomposition 
    \[\epsilon_{F,\chi}=\sum_{i=1}^{v_{F,\chi}} \sum_{j=1}^{n_\eta} \f^{(j)}_{F,\eta_i}.\] 
    The first claim now follows from \cref{skew-field-general} and the Jordan--Hölder theorem. The second assertion follows by using $n_{F,\chi} s_{F,\chi}=\chi(1)=\eta(1) w_\chi=n_\eta s_\eta w_\chi$ (here the second equality comes from \eqref{eq:res-z-eta}) and $[F(\eta):F_\chi]=w_\chi/v_{F,\chi}$ (\cref{cyclicity}).
\end{proof}

\begin{theorem} \label{Wedderburn-general}
    Let $\chi\in\Irr(\G)$, and let $\eta \mid \res^\G_H\chi$ be an irreducible constituent. Then there are isomorphisms
    \[{D_{F,\chi}} \xrightarrow[\xi]{\sim} \bigoplus_{\mathbf 0\le \kk < \ff} D_{W,\chi} \ggamma^{\prime\prime}_{F,\eta,\kk} \,\mathop{\simeq}_{(*)}\, \bigoplus_{\substack{\mathbf 0\le \ll < p^{\mathbf n_0} \\ \ll\equiv 0\mod{\vv_{F,\chi} }}} \tD_\eta \ggamma^{\prime\prime}_{F,\eta,\ll} \xrightarrow[\rho]{\sim} {\Quot\left(\OO_{D_\eta}[[\X;\ttau_F,\ddelta_F]]\right)} ,\]
    where conjugation by $\ggamma^{\prime\prime}_{F,\eta,\ll}$ acts as $\ttau_F^\ll$.
    On the respective centres, these isomorphisms restrict to
    \[\cent\left(D_{F,\chi}\right) \xleftarrow{\sim} {\Q^{F_\chi}\left(\Gamma_{F,{\w_\chi}/{\vv_{F,\chi}},\eta}^{\prime\prime}\right)} \xrightarrow{\sim} {\Frac\left(\OO_{F_\chi}[[\T]]\right)},\]
    where $T_i=(1+X_i)^{w_{\chi,i}/v_{F,\chi,i}}-1$.
\end{theorem}
\begin{proof}
    We first establish the isomorphism $\rho$. The map is the identity on $D_\eta$, and on $\ggamma^{\prime\prime}_{F,\eta,\ll}$ it is given by 
    \[\rho\left(\ggamma^{\prime\prime}_{F,\eta,\ll}\right)\colonequals (1+\X)^\ll=\prod_{i=1}^d (1+X_i)^{\ell_i}.\] 
    The property of being a homomorphism is easily checked, and surjectivity is immediate. To show injectivity, first observe that $\Q^{F_\chi}\big(\Gamma^{\prime\prime}_{F,{\w_\chi}/{\vv_{F,\chi}},\eta}\big)$ is in the centre of the direct sum in the statement: indeed, each $\tau_i$ acts trivially on $F_\chi$, so coefficients in $F_\chi$ commute with everything, and conjugation by $\ggamma_{F,{\w_\chi}/{\vv_{F,\chi}},\eta}^{\prime\prime}$ acts on coefficients in $D_\eta$ as $\ttau^{{\w_\chi}/{\vv_{F,\chi}}}$, that is to say, trivially. The direct sum above has dimension $(w_\chi/v_{F,\chi})^2 s_\eta^2$ over $\Q^{F_\chi}\big(\Gamma^{\prime\prime}_{F,\w_\chi/\vv_{F,\chi},\eta}\big)$. By \cref{skew-power-series-centre}, the total ring of quotients of the skew power series ring has centre $\Frac\left(\OO_{F_\chi}[[\T]]\right)$, so its index is $(w_\chi/v_\chi^F)^2 s_\eta^2$. Since $\rho$ restricts to a nonzero surjective homomorphism $\Q^{F_\chi}\big(\Gamma^{\prime\prime}_{F,\w_\chi/\vv_{F,\chi},\eta}\big)\to\Frac\left(\OO_{F_\chi}[[\T]]\right)$ of fields, this restriction must be an isomorphism. Equality of dimensions now shows that $\Q^{F_\chi}\big(\Gamma^{\prime\prime}_{F,\w_\chi/\vv_{F,\chi},\eta}\big)$ is the entire centre, and that $\rho$ is an isomorphism, as claimed.

	The isomorphism $(*)$ follows by combining \cref{item:W-t-r-ii}, which is applicable since $F(\eta)/W$ is totally ramified, with the base change statement \cref{inertia-direct-product}.

    We turn to the isomorphism $\xi$; the proof follows along the lines of \cite[Corollary~4.3]{W2}.
    By a dimension counting argument as in loc.cit., any injective ring homomorphism $\xi$ is necessarily an isomorphism.
    Using \cref{skew-field-general} and \cite[Corollary~4.2]{W2}, we see that after possibly twisting $\chi$ by a type~W character, we may assume that $F=F_\chi$, so that there are isomorphisms
    \[D_{F,\chi}\simeq \f^{(j)}_{F_\chi,\eta} \Q^{F_\chi}(\G) \f^{(j)}_{F_\chi,\eta}.\]
    Now define $\xi$ by setting
    \[\xi\left(\f^{(j)}_{F_\chi,\eta} x \f^{(j)}_{F_\chi,\eta}\right) \colonequals \sum_{\mathbf 0\le \ii <\ff} f^{(j)}_{W,\eta} x \ttau_F^\ii\left(f^{(j)}_{W,\eta}\right) \cdot \ggamma^{\prime\prime}_{F,\eta,\ii}\]
    for all $x\in \Q^{F_\chi}(\G)$. Here the coefficient of $\ggamma^{\prime\prime}_{F,\eta,\ii}$ is seen as an element of $D_{W,\chi}$ under the isomorphism
    \[f^{(j)}_{W,\eta} \Q^W(\G) \ttau_F^\ii\left(f^{(j)}_{W,\eta}\right) \simeq D_{W,\chi}\]
    given by \cref{item:W-t-r-ii}. The fact that $\xi$ is an injective ring homomorphism is verified in the same way as in \cite[Corollary~4.3]{W2}, using the $d$-dimensional orthogonality relation (\cref{f-orthogonalities}).
\end{proof}

\begin{remark}[Dyadic Wedderburn decomposition]
	In our discussion above, we assumed $p$ to be an odd prime. The case of $p=2$ for one-dimensional $\G$ was discussed in \cite[\S6]{W2}. As explained there, the main difficulty in the $p=2$ case is extending automorphisms from $\Gal(F(\eta)/F_\chi)$ to $D_\eta$. If $s_\eta=1$, then this is trivial, and similarly to \cite[Proposition~6.1]{W2}, our \cref{Wedderburn-general} remains valid. If $s_\eta=2$, which is the only other possibility, then the Wedderburn components have not been described even when $d=1$.
\end{remark}

\section{Towards applications in Iwasawa theory} \label{sec:Iwasawa}

In the case $d=1$, a description of the Wedderburn decomposition of $\Q(\G)$ has been fruitfully applied to prove integrality properties of characteristic elements in Iwasawa theory \cite{FM}; the techniques used trace back to the work of Nichifor--Palvannan \cite{NichiforPalvannan}. In the presence of an Iwasawa main conjecture, this translates to integrality statements of $p$-adic $L$-functions \cite{EpAC}. 

In this section, we establish the corresponding algebraic result for $d\ge1$.

\subsection{Integrality in multivariate skew power series rings}
We return to the setup of \cref{sec:skew-power-series}, so $\mathfrak O= \OO_D[[\X; \ttau,\ddelta]]$ is a multivariate skew power series ring, and $\mathfrak D$ is its skew field of quotients.

The original form of the following integrality statement appeared in \cite[Proposition~2.13]{NichiforPalvannan}, for power series rings $\mathfrak O=\OO_D[[X_1]]$, that is, for $1$-dimensional $p$-adic Lie groups $\G=H\times\ZZ_p$ that are a direct product of a finite group $H$ and $\ZZ_p$.
In \cite[\S6]{EpAC}, it was generalised to skew power series rings $\mathfrak O=\OO_D[[X_1;\tau_1,\delta_1]]$, that is, to semidirect products $\G=H\rtimes\ZZ_p$. As we will see, it generalises to higher dimension by using the same techniques.

Let $\mathfrak D^{\times,\ab}$ denote the abelianisation of the group of invertible elements in $\mathfrak D$. Recall that for $m\ge1$, the Dieudonné determinant is the unique group homomorphism $\det:\GL_m(\mathfrak D)\to \mathfrak D^{\times,\ab}$ characterised by the following properties. The Dieudonné determinant of an elementary matrix (that is, a matrix differing from the identity matrix by a single off-diagonal entry) is one, and $\det(\diag(x,1,\ldots,1))=[x]$, where $[x]\in\mathfrak D^{\times,\ab}$ is the equivalence class of $x\in\mathfrak D^\times$. For more on the Dieudonné determinant, we refer to \cite[165--166]{CR}.

\begin{proposition} \label{integrality}
    For $m\ge 1$, any matrix $A\in M_m(\mathfrak O) \cap \GL_m(\mathfrak D)$ has Dieudonné determinant 
    \[\det(A)\in \im\left(\mathfrak O \cap \mathfrak D^\times \to \mathfrak D^{\times,\ab}\right).\]
\end{proposition}
\begin{proof}
    Using the identification \eqref{eq:iterated-skew-power-series}, view $\mathfrak O$ as an iterated skew power series ring: writing 
    \[\mathfrak o\colonequals \OO_D[[X_1,\ldots,X_{d-1};\tau_1,\ldots,\tau_{d-1};\tau_1-1\ldots,\tau_{d-1}-1]],\] 
    we have $\mathfrak O=\mathfrak o[[X_d,\tau_d,\tau_d-1]]$.
    Note that $\mathfrak o$ is a local ring with maximal ideal 
    \[\mathfrak m=(\pi_D,X_1,\ldots,X_{d-1})\subset \mathfrak o\] 
    by \cite[Proposition~2.11]{Venjakob-WPT}, and $\mathfrak o$ is complete and separated with respect to its $\mathfrak m$-adic topology. In particular, it is possible to apply Weierstraß theory as in \cite[\S3]{Venjakob-WPT} in $\mathfrak O$ with respect to the variable $X_d$. For a skew power series $f=\sum_{i=0}^\infty a_i X_d^i\in\mathfrak O$ with $a_i\in\mathfrak o$, let 
    \[\orr_{d}(f)\colonequals\inf\{i\ge0: a_i\in\mathfrak o^\times\}\] 
    denote the reduced order with respect to $X_d$.
    
    Let $S\subset \mathfrak O$ be the multiplicatively closed set generated by $\pi_D,X_1,\ldots,X_{d-1}$. We show that $S^{-1}\mathfrak O$ is a principal ideal domain, that is, all left ideals and all right ideals are principal.
    Let $I\subseteq S^{-1}\mathfrak O$ be a nonzero left ideal (the proof is analogous for right ideals). Then there is an element $f\in I\cap \mathfrak O$ of minimal reduced $X_d$-order. By Weierstrass preparation, $f=\epsilon F$ where $\epsilon\in\mathfrak O^\times$ and $F\in\mathfrak O$ is a distinguished skew polynomial in $X_d$ over $\mathfrak o$. We show $S^{-1}\mathfrak O F=I$. The principal ideal is obviously contained in $I$. For the converse, apply Weierstraß division with respect to $F$:
    \[\mathfrak O=\mathfrak O F\oplus\bigoplus_{i=0}^{\orr_{d} F-1} \mathfrak o X_d^i,\]
    and tensor with $(S\cap\mathfrak o)^{-1}\mathfrak o$:
    \[S^{-1}\mathfrak O=S^{-1}\mathfrak O F\oplus\bigoplus_{i=0}^{\orr_{d} F-1} (S\cap\mathfrak o)^{-1}\mathfrak o X_d^i.\]
    Let $g\in I$ be arbitrary, and write $g=hF+r$ under this decomposition. If $r\ne 0$, then let $x\in S$ be an element such that $xr\in\mathfrak o[X_d]-\mathfrak m\mathfrak o[X_d]$. Then $\orr_d(xr)\le \deg_{X_d}(xr)=\deg_{X_d}(r)<\orr_d(F)$, so $xr\in I\cap\mathfrak O$ has lower reduced order than $f$, contradicting the minimality of $\orr_d(f)$. Hence $r=0$, and $S^{-1}\mathfrak O F=I$.

    Now we can apply Jacobson's elementary reduction theorem \cite[Chapter~3, Theorem~16]{Jacobson}: a matrix $A$ as in the statement can be decomposed as a product $A=UBV$ where $U,B,V\in M_m(S^{-1}\mathfrak O)$, $B$ is diagonal, and $U$ and $V$ are products of elementary, permutation, and scalar matrices. For $?\in\{U,B,V\}$, let $x_?\in \mathfrak O$ such that $x_? ? \in M_m(\mathfrak O)-M_m(\mathfrak m\mathfrak O)$. By Weierstraß preparation, we have $\det(?)=\det(x_?^{-1}\cdot x_? ?)=\det(x_?)^{-1} [\beta_?] [J_?]$, where $\beta_?\in\mathfrak O^\times$ is a unit, $J_?\in \mathfrak o[X_d]$ is a monic skew polynomial, and $[-]$ denotes equivalence classes in $\mathfrak D^{\times,\ab}$. As $U$ and $V$ are products of elementary, permutation, and scalar matrices, we have $J_U=J_V=1$, and therefore
    \begin{equation} \label{eq:det-A}
        \det(A)=\det(x_U x_B x_V)^{-1}\cdot[\beta_U\beta_B\beta_V]\cdot [J_B].
    \end{equation}
    Recall that $\nr_{M_m(\mathfrak D)/\cent(\mathfrak D)}=\nr_{\mathfrak D/\cent(\mathfrak D)}\circ\det$, see \cite[(7.42)]{CR}. Therefore
    \begin{equation} \label{eq:nr-A}
        \nr_{M_m(\mathfrak D)/\cent(\mathfrak D)}(A)=\nr_{\mathfrak D/\cent(\mathfrak D)}(x_U x_B x_V)^{-1}\cdot\nr_{\mathfrak D/\cent(\mathfrak D)}(\beta_U\beta_B\beta_V)\cdot \nr_{\mathfrak D/\cent(\mathfrak D)}(J_B).
    \end{equation}
    It follows from \cite[Theorem 10.1]{MO} that $\nr_{M_m(\mathfrak D)/\cent(\mathfrak D)}(A)\in\OO_{\cent(\mathfrak D)}$. Since each $\beta_?$ is a unit, we have $\nr_{\mathfrak D/\cent(\mathfrak D)}(\beta_U\beta_B\beta_V)\in\OO_{\cent(\mathfrak D)}^\times$. Using the splitting map constructed in \cref{splitting}, an argument similar to \cite[Lemma~6.6]{EpAC} shows that $\nr_{\mathfrak D/\cent(\mathfrak D)}(J_B)$ is a monic skew polynomial in $X_d$. Since the left hand side of \eqref{eq:nr-A} is in $\OO_{\cent(\mathfrak D)}$, this forces $\nr_{\mathfrak D/\cent(\mathfrak D)}(x_U x_B x_V)^{-1}\in \OO_{\cent(\mathfrak D)}$, which in turn shows $\det(x_U x_B x_V)^{-1}\in\im(\mathfrak O^\times\to\mathfrak D^{\times,\ab})$. The claim now follows from \eqref{eq:det-A}.
\end{proof}

\subsection{Integrality of characteristic elements}
Consider the connecting homomorphism $\partial: K_1(\Q(\G))\to K_0(\Lambda(\G),\Q(\G))$ in relative $K$-theory, as defined e.g. in \cite[29ff.]{SujathaKtheory}.
Let $\Gamma_0$ be as in \cref{Gamma0}, and let $\Lambda(\Gamma_0)$ be the associated Iwasawa algebra: then any $\Lambda(\G)$-module is also a $\Lambda(\Gamma_0)$-module.

Let $X$ be a finitely generated $\Lambda(\G)$-module that is torsion over $\Lambda(\Gamma_0)$ and has projective dimension $\pd_{\Lambda(\G)} X \le 1$. Then there is a projective resolution $P\hookrightarrow \Lambda(\G)^n\twoheadrightarrow X$ for some projective $\Lambda(\G)$-module $P$. Since $X$ is $\Lambda(\Gamma_0)$-torsion, the image of $[\Lambda(\G)^n]-[P]$ under the natural map $K_0(\Lambda(\G))\to K_0(\Q(\G))$ is zero. As this map is injective \cite[Corollary 3.8]{Witte}, there is a projective $\Lambda(\G)$-module $Q$ such that $P\oplus Q\simeq \Lambda(\G)^n\oplus Q$. By possibly enlarging $Q$, this leads to a free resolution $\Lambda(\G)^m\xhookrightarrow{A}\Lambda(\G)^m\twoheadrightarrow X$, where the first arrow is multiplication by $A\in M_m(\Lambda(\G))\cap \GL_m(\Q(\G))$. This defines an element $[(X,A)]\in K_0(\Lambda(\G),\Q(\G))$ in the relative $K_0$-group. By a characteristic element for $X$ we mean an element $\xi_X \in K_1(\Q(\G))$ such that $\partial(\xi_X ) =[(X,A)]$.

Let $R$ be noetherian integral domain, and let $\mathcal A$ be a separable algebra over a field that is a separable extension of $\Frac(R)$. An $R$-order $\mathfrak M$ in $\mathcal A$ is called a graduated order if there are indecomposable idempotents $e_1,\ldots,e_t\in\mathfrak M$ such that $\sum_{i=1}^t e_i=1$, and $e_i \mathfrak M e_i\subset e_i \mathcal A e_i$ is a maximal $R$-order for all $i=1,\ldots,t$. We are interested in the case $R=\Lambda(\Gamma_0)$ and $\mathcal A=\Q(\G)$. Note that all maximal $\Lambda(\Gamma_0)$-orders in $\Q(\G)$ are graduated \cite[Remark~2.7]{graduated}.

\begin{proposition} \label{integrality-of-char-elements}
    If $X$ is a finitely generated $\Lambda(\G)$-module that is torsion over $\Lambda(\Gamma_0)$ and has projective dimension $\pd_{\Lambda(\G)} X \le 1$, and $\xi_X$ is a characteristic element, then for every graduated order $\mathfrak M$ of $\Q(\G)$ containing $\Lambda(\G)$, we have
    \[\xi_X\in \im\left(\mathfrak M\cap \Q(\G)^\times \to K_1(\Q(\G))\right).\]
\end{proposition}
\begin{proof}
    The proof is essentially the same as in \cite[Corollary~4.3]{graduated}.
    We may work Wedderburn component-wise. Then $\mathfrak M \epsilon_\chi$ is a graduated order in $\Q(\G)\epsilon_\chi$. Fix a maximal order $\Omega_\chi$ in $D_\chi$. The same proof as in \cite[Proposition~2.12]{graduated} shows that $\mathfrak M \epsilon_\chi$ is isomorphic to a so-called standard form graduated order; for our purposes, it suffices to know that this is an order of the form
    \[\begin{psmallmatrix}
        \Omega_\chi & * & \dots & * \\
        * & \Omega_\chi & \dots & * \\
        \vdots & \vdots & \ddots & \vdots \\
        * & * & \dots & \Omega_\chi
    \end{psmallmatrix},\]
    where the entries $*$ come from certain two-sided ideals in $\Omega_\chi$. (See \cref{rem:graduated-higher} below on why the proof of \cite[Proposition~2.12]{graduated} is applicable.) Now by the same methods as in \cite[Corollary~4.3]{graduated}, the proof reduces to an application of \cref{integrality}.
\end{proof}

\begin{remark} \label{rem:graduated-higher}
    The theory of graduated orders over regular local rings of dimension at most two was developed in \cite{graduated}. The main application of the results in that work is for Iwasawa algebras over one-dimensional $p$-adic Lie groups $\G=\ZZ_p\ltimes H$. In most of \cite{graduated}, including the result on standard forms, the only reason for the upper bound on the dimension is to be able to use Ramras's result uniqueness of maximal orders up to conjugacy (\cref{Ramras}). This, as we pointed out after \cref{Ramras}, is still valid in higher dimension, and therefore many of the results of \cite{graduated} generalise to the setting of the present work.
\end{remark}

\begin{remark}
	Let us relate the characteristic element defined above to the one considered in \cite{CFKSV} and subsequent works. Let $F_\infty/F$ be a Galois extension such that 
	\begin{enumerate}[label*=(\roman*)]
		\item \label{F-cyc-contain} $F_\infty$ contains the cyclotomic $\ZZ_p$-extension $F_{\cyc}$ of $F$, 
		\item \label{unramified} $F_\infty/F$ is unramified outside a finite set of places,
		\item \label{G-form} $\G\colonequals\Gal(F_\infty/F)$ is of the form $\G\simeq H\rtimes\Gamma$ where $H$ is a finite group and $\Gamma\simeq\ZZ_p^d$ for some $d\ge1$, and
		\item \label{H-coprime} $\G$ has no $p$-torsion.
	\end{enumerate} 
	Let us write $\HH\colonequals\Gal(F_\infty/F_\cyc)\simeq H\rtimes \Gamma_\HH$, where $\Gamma_\HH\simeq\ZZ_p^{d-1}$. The relationships between these fields and their Galois groups are summarised by the following diagram.
	
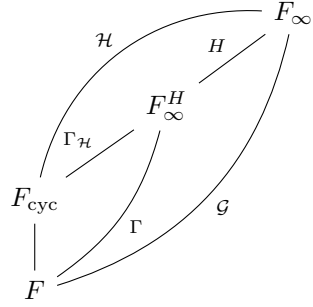
\[\begin{tikzcd}[every arrow/.style={draw,no head}]
		&& {F_\infty} \\
		& F_\infty^H \\
		{F_\cyc} \\
		F
		\arrow["H"', from=1-3, to=2-2]
		\arrow["\G", from=1-3, to=4-1, bend left=30]
		\arrow["{\Gamma_\HH}"', from=2-2, to=3-1]
		\arrow["\Gamma", from=2-2, to=4-1, bend left=20]
		\arrow["\HH"', from=1-3, to=3-1, bend right=40]
		\arrow[""', from=3-1, to=4-1]
	\end{tikzcd}\]
	As in \cite{CFKSV}, let $S$ be the set of elements $x\in\Lambda(\G)$ such that $\Lambda(\G)/x\Lambda(\G)$ is finitely generated over $\Lambda(\HH)$, and let $S^*\colonequals\bigcup_{n\ge0} p^n S^*$. Let $\mathfrak M_\HH(\G)$ denote the category of finitely generated $\Lambda(\G)$-modules that are $S^*$-torsion. If $X\in\mathfrak M_\HH(\G)$, then a characteristic element $\xi_X^*$ is defined to be a lift of the class of $X$ under the connecting homomorphism $K_1(\Lambda(\G)_{S^*})\to K_0(\Lambda(\G),\Lambda(\G)_{S^*})$.
	
	In the definition of $\Q(\G)$, we invert all regular elements, as opposed to only those in $S^*$. Consequently, there is a commutative square, with the vertical maps being the natural maps induced by $\Lambda(\G)_{S^*}\to \Q(\G)$:
	\[\begin{tikzcd}
		\xi_X^*\in K_1(\Lambda(\G)_{S^*}) \ar[r,"\partial"] \ar[d] & K_0(\Lambda(\G),\Lambda(\G)_{S^*}) \ar[d] \\
		\xi_X\in K_1(\Q(\G)) \ar[r, "\partial"] & K_0(\Lambda(\G),\Q(\G)) 
	\end{tikzcd}\]
	
	If $X$ is a finitely generated $\Lambda(\G)$-module that is torsion over $\Lambda(\Gamma_0)$ and has projective dimension $\pd_{\Lambda(\G)} X \le 1$, and $X\in\mathfrak M_\HH(\G)$, then both notions of characteristic element make sense, and one may choose $\xi_X$ to be the image of $\xi_X^*$ under the left vertical map in the diagram.
\end{remark}

\printbibliography
\end{document}